\documentclass[a4paper,twoside]{article}
\usepackage{geometry}
\usepackage{fancyhdr}
\usepackage{amsmath}
\usepackage{amssymb}
\usepackage{amsthm}
\usepackage{enumerate}
\usepackage{xcolor}

\newtheorem{thm}{Theorem}[section]

\newtheorem{prop}{Proposition}[section]
\newtheorem{rem}{Remark}[section]
\newtheorem{defi}{Definition}[section]

\newcommand{\eps}{\varepsilon}
\newcommand{\R}{\mathbb{R}}

\newcommand{\N}{\mathbb{N}}
\newcommand{\Z}{\mathbb{Z}}

\renewcommand{\div}{{\rm div}\,}
\newcommand{\Id}{{\rm Id}\,}

\newcommand\ds{\displaystyle}

\newcommand{\Int}{\displaystyle \int}

\def\d{\partial}

\def\ddj{\dot \Delta_j}
\def\dj{\Delta_j}

\def\tilde{\widetilde}

\def\div{{\rm div}\,}

\def\cC{{\mathcal C}}
\def\cF{{\mathcal F}}

\def\cP{{\mathcal P}}
\def\cQ{{\mathcal Q}}
\def\cS{{\mathcal S}}
\def\dt{\delta\!\theta}
\def\dQ{\delta\!Q}

\def\du{\delta\!u}

\begin{document}

\title{On the well-posedness of the full low-Mach number limit system 
in general critical Besov spaces}
\author{Rapha\"{e}l DANCHIN and Xian LIAO\thanks{Universit\'{e} Paris-Est Cr\'{e}teil, LAMA, UMR 8050, 61 Avenue de G\'{e}n\'{e}ral de Gaulle, 94010 Cr\'{e}teil Cedex, France}}
\date\today

\maketitle

\begin{abstract}
This work is devoted to  the well-posedness issue for the  low-Mach number
 limit system obtained from the full compressible
 Navier-Stokes system, in the whole space $\R^d$ with $d\geq2.$
 
 In the case where the initial temperature (or density) is close to  a positive  constant, we establish the local existence and
 uniqueness of a solution in critical homogeneous Besov spaces
 of type $\dot B^s_{p,1}.$ If, in addition,  the initial velocity is small then
 we show that the solution exists for all positive time. 
 In the fully nonhomogeneous case,
 we establish the local well-posedness in nonhomogeneous Besov spaces
 $B^s_{p,1}$ (still with critical regularity)  for 
 arbitrarily large data with positive initial temperature. 
  
  Our analysis strongly relies on the use of a \emph{modified} divergence-free velocity which
  allows to reduce the system to a nonlinear coupling between 
  a parabolic equation and some evolutionary Stokes system. 
  As in the recent work by Abidi-Paicu \cite{AbidiP07} concerning the
 density dependent incompressible Navier-Stokes equations,
  the Lebesgue exponents of the Besov spaces 
 for the temperature and the (modified) velocity, need not be the same. 
 This enables us to consider initial data  in Besov spaces
 with a  \emph{negative} index of regularity.  
 \end{abstract}

\section{Introduction}
\setcounter{equation}{0}

The full Navier-Stokes system
\begin{equation}\label{eq:fullnewtonian}\left\{
\begin{array}{ccc}
\d_t\rho+\div(\rho v)&=&0,\\
\d_t(\rho v)+\div(\rho v\otimes v)-\div \sigma+\nabla p&=&0,\\
\d_t(\rho e)+\div(\rho v e)-\div(k\nabla\vartheta)+p\,\div v&=&\sigma\cdot Dv,
\end{array}
\right.
\end{equation}
governs the free evolution  of a viscous and heat conducting
compressible Newtonian fluid. 

In the above system,  $\rho=\rho(t,x)$
 stands for the mass density, $v=v(t,x),$ for  the velocity field and $e=e(t,x),$
for  the internal energy per unit mass. The time variable $t$
belongs to $\R^+$ or to $[0,T]$ and the space variable $x$ is in $\R^d$
with $d\geq2.$ 
The scalar functions  $p=p(t,x)$ and $\vartheta=\vartheta(t,x)$ denote the pressure and temperature respectively, and  $\sigma$ is  the viscous strain tensor,  given by
$$
\sigma=2\zeta Sv+\eta\, \div v \,\Id,
$$
where  $\Id$ is the $d\times d$ identity matrix, 
$Sv:=\frac12(\nabla v+Dv),$  the so-called deformation 
tensor\footnote{In all the paper, we agree that for $v=(v^1,\cdots,v^d)$ a vector field, 
 $(Dv)_{ij}:=\d_jv^i$ and $(\nabla v)_{ij}:=\d_iv^j.$}.
  The  heat conductivity $k$ and the Lam\'e (or viscosity) 
 coefficients  $\zeta$ and $\eta$ may depend smoothly on $\rho$ and on $\vartheta.$
 The above system has to be supplemented with  two {\it state equations}
involving $p$,~$\rho$,~$e$ and~$\vartheta.$

\bigbreak

In the present paper, we want to consider  the low Mach number limit of the full Navier-Stokes System 
\eqref{eq:fullnewtonian}. {}From an heuristic viewpoint, this  amounts to
neglecting  the compression due to pressure variations,  a common assumption when describing highly subsonic flows.  Following   the introduction of P.-L. Lions's book \cite{PLions96}
(see also the physics book by Zeytounian \cite{Zeytounian04}), we here
explain  how  the low Mach number limit system may be derived \emph{formally}
from \eqref{eq:fullnewtonian}.
 For simplicity we restrict ourselves to the case of an \emph{ideal} gas, namely we assume that
\begin{equation}\label{eq:state}
p=R\rho \vartheta,\quad e=C_v\vartheta,
\end{equation}
where $R,C_v$ denote the ideal gas constant and the specific heat constant, respectively. 
\smallbreak
Let us define   the  (dimensionless) 
 Mach number  $\eps$ as the ratio of the reference velocity 
 over the reference sound speed of the fluid, 
 then suppose that $(\rho, v, \vartheta)$ is some given  classical solution of System \eqref{eq:fullnewtonian} corresponding to the small  $\eps.$ 
Then the rescaled triplet
$$
\left(\rho_\eps(t,x)=\rho\Bigl(\frac{t}{\eps},x\Bigr),\quad 
v_\eps(t,x)=\frac{1}{\eps}v\Bigl(\frac{t}{\eps},x\Bigr),\quad 
\vartheta_\eps(t,x)=\vartheta\Bigl(\frac{t}{\eps},x\Bigr)\right)
$$
satisfies
\begin{equation}\label{eq:epsilon}\left\{
\begin{array}{ccc}
\d_t\rho_\eps+\div(\rho_\eps v_\eps)&=&0,\\[1ex]
\d_t(\rho_\eps v_\eps)+\div (\rho_\eps v_\eps\otimes v_\eps)-\div\sigma_\eps+\frac{\nabla p_\eps}{\eps^2}&=&0,\\[1ex]
\frac{1}{\gamma-1}(\d_tp_\eps+\div(p_\eps v_\eps ))-\div(k_\eps\nabla \vartheta_\eps)+p_\eps\div v_\eps&=&\eps^2\sigma_\eps\cdot Dv_\eps,
\end{array}
\right.
\end{equation}
with
$$
\displaylines{
\sigma_\eps=2\zeta_\eps Sv_\eps+\eta_\eps \div v_\eps \Id,\quad p_\eps=R\rho_\eps\vartheta_\eps,
\quad \gamma=1+{R}/{C_v},\cr
\zeta_\eps=\frac{1}{\eps}\zeta\Bigl(\frac{t}{\eps},x\Bigr),\quad
\eta_\eps=\frac{1}{\eps}\eta\Bigl(\frac{t}{\eps},x\Bigr)\ \hbox{ and }\ 
k_\eps=\frac{1}{\eps}k\Bigl(\frac{t}{\eps},x\Bigr).}
$$

By letting the Mach number $\eps$ go to $0$, the momentum equation of (\ref{eq:epsilon}) implies that
$$
p_\eps=P(t)+\Pi(t,x)\eps^2+o(\eps^2).
$$
Plugging this formula into the energy equation of (\ref{eq:epsilon}) entails that $P(t)$ is independent of $t$, provided $v_\eps$ and $\nabla\vartheta_\eps$ vanish at infinity.
{}From now on, we shall denote this constant by $P_0.$ 

Bearing in mind the equation of state given in \eqref{eq:state}, 
we deduce that  $\rho=P_0/(R\vartheta).$
Therefore, denoting $C_p=\gamma C_v=\gamma R/(\gamma-1),$   the low Mach number limit system reads
\begin{equation}\label{eq:original Mach limit}
\left\{ \begin{array}{ccc}
\rho C_{P} (\d_{t}\vartheta+ v \cdot \nabla\vartheta)-\div( k\nabla\vartheta ) & =&0,\\
\rho ( \d_{t}v+v\cdot\nabla v )-\div\sigma + \nabla\Pi & = & 0, \\
\gamma P_0 \div v - (\gamma -1)\div( k \nabla\vartheta)& = &0.
\end{array} \right.
\end{equation}

A number of mathematical results concerning  the low Mach number limit from (\ref{eq:epsilon}) to (\ref{eq:original Mach limit}) have been obtained in the past three decades, most
of them being dedicated to   the isentropic  or to the  barotropic isothermal cases
(that is   $\vartheta\equiv const$ and $p=p(\rho)$)
under the assumption that the viscosity coefficients are independent of $\rho.$
{}From a technical viewpoint, those  latter cases are easier to deal with inasmuch as  only 
the first two equations of System \eqref{eq:epsilon} have to be considered. 
As a consequence,  the expected limit system is  the standard incompressible
 Navier-Stokes or Euler  system.
In the inviscid case,  the first mathematical 
results concerning the incompressible limit go back to the eighties with the works by  Klainerman and Majda \cite{KlainermanM82},  Isozaki \cite{Isozaki87,Isozaki89} and Ukai \cite{Ukai86}.  
In the viscous case, the justification of the \emph{incompressible Navier-Stokes} equation as  the zero-Mach limit of the \emph{compressible Navier-Stokes} equation, has been 
done in different contexts by  e.g.  Danchin \cite{Danchin02perio,Danchin02limit},  Desjardins and Grenier \cite{DesjardinsG99}, Hagstr\"om and Lorentz \cite{HL}, Hoff \cite{Hoff98}, Lions \cite{PLions96vol2}, and Lions and Masmoudi \cite{PLionsMasmoudi98}. 
In contrast, there are quite a few results for  the full  system. The inviscid case -- the \emph{non-isentropic Euler}, 
has been considered in the whole space and periodic settings 
by   M\'{e}tivier and Schochet \cite{MetivierS01,MetivierS03}. 
  Recently  Alazard  performed a rigorous analysis for the \emph{full Navier-Stokes} equations with large temperature variations in the \emph{Sobolev} spaces $H^s$ with $s$ large enough, see \cite{Alazard06}. 
  In the framework of   variational solutions with finite energy, this asymptotics has been 
  justified in the somewhat different regime of Oberbeck-Boussinesq approximation 
  (see the book   \cite{FeireislNovotny} by Feireisl-Novotn\'y).

To our knowledge, the only work dedicated to
 the limit system \eqref{eq:original Mach limit} in the general case
where $\vartheta$ is not a constant or the conductivity $k$ is not zero
(note that if $k\equiv0$ then  the system reduces to the nonhomogeneous incompressible
Navier-Stokes equations studied in e.g. \cite{AbidiP07,Danchin03})
is the paper \cite{Embid} by P. Embid. There, local-in-time  existence of smooth solutions 
(in Sobolev spaces) is established
not only for \eqref{eq:original Mach limit} but also for a more complicated system
of reacting flows.

The present  paper is to study the \emph{well-posedness} issue for the full low Mach number
limit system (\ref{eq:original Mach limit}) in the critical Besov spaces, locally and globally.
We expect our work  to be the first step of justifying rigorously the limit process  in the critical Besov spaces, a study that we plan to do in the future.

\medbreak
In what follows, we assume
that  the coefficients $(\zeta,\eta,k)$ in \eqref{eq:original Mach limit}
are $C^\infty$ functions of the temperature $\vartheta$, 
and we consider only the viscous and heat-conducting case, namely 
$$
k(\vartheta)>0,\qquad  \zeta(\vartheta)>0\qquad \textrm{  and  }\qquad \eta(\vartheta)+2\zeta(\vartheta)>0.
$$
At first sight, this assumption ensures that
System \eqref{eq:original Mach limit}  is of parabolic type, up 
to the pressure term that may be seen as  the Lagrange multiplier
corresponding to the constraint given in the last equation.
Handling this relation   between $\div v$
and the temperature is the first difficulty that has to be faced. 
In order to reduce the study to a system which looks more like
the incompressible Navier-Stokes equations, it is 
natural to perform the following change of velocity:
\begin{equation}\label{eq:newvelocity}
u=v-\alpha k\nabla\vartheta\quad\hbox{with}\quad
\alpha:=\frac{\gamma-1}{\gamma P_0}=\frac{R}{C_p P_0}=\frac{1}{C_p \rho\vartheta}\cdotp
\end{equation}
We claim  that $(\vartheta,v)$ satisfies  \eqref{eq:original Mach limit}
(for some $\nabla\Pi$) if and only there exists some function $Q$ so that 
$(\vartheta,u)$ fulfills 
\begin{equation}\label{equation:local}\left\{
\begin{array}{ccc}
\d_t\vartheta+u\cdot\nabla\vartheta-\div(\kappa\nabla\vartheta)&=&f(\vartheta),\\
\d_t u+u\cdot\nabla u-\div(\mu\nabla u)+\vartheta\nabla Q&=&h(\vartheta,u),\\
\div u&=&0,
\end{array}
\right.\end{equation}
with 
$\kappa(\vartheta)=\alpha k \vartheta,$ 
$\mu(\vartheta)=\alpha C_p \zeta\vartheta,$
$f(\vartheta)=-2\alpha k|\nabla\vartheta|^2$
and
\begin{align*}
h(\vartheta,u)=A_1|\nabla\vartheta|^2\nabla\vartheta
+A_2\Delta\vartheta\nabla\vartheta+A_3\nabla\vartheta\cdot\nabla ^2\vartheta
+A_4\nabla u\cdot\nabla\vartheta+A_5Du\cdot\nabla\vartheta,
\end{align*}
where the coefficients $A_i$ ($i=1,\cdots 5$) are functions of $\vartheta$
the value of which is given in \eqref{eq:A} below. 
\medbreak
In order to derive \eqref{equation:local}, we first notice that 
$$
\begin{array}{lll}
\d_t(\rho v)+\div(\rho v\otimes v)&=&\d_t(\rho u)+\div(\rho v\otimes u)
+\d_t(\rho\alpha k\nabla \vartheta)+\div(\rho\alpha kv\otimes\nabla\vartheta),\\[1ex]
&=&\rho(\d_tu+v\cdot\nabla u)+\d_t(C_p^{-1}\vartheta^{-1}k\nabla\vartheta)
+\div(C_p^{-1}\vartheta^{-1}kv\otimes\nabla\vartheta).
\end{array}
$$
Hence, given that $k$ is a function of $\vartheta$ we deduce that
there exists some function $Q_1$ so that 
$$
\d_t(\rho v)+\div(\rho v\otimes v)=\rho(\d_tu+u\cdot\nabla u)+\rho\alpha kDu\cdot\nabla\vartheta
+\div(C_p^{-1}\vartheta^{-1}kv\otimes\nabla\vartheta)+\nabla Q_1.
$$
Next, using the fact that
$$
Sv=\frac 12(\nabla u-Du)+Dv,
$$
we may write that
$$
\begin{array}{lll}
-\div\sigma&=&-\div(\zeta\nabla u)+\div(\zeta Du)-2\div(\zeta Dv)-\nabla(\eta\div v),\\[1ex]
&=&-\div(\zeta\nabla u)+\nabla u\cdot\nabla\zeta+2\div(v\otimes\nabla\zeta)-\nabla(\eta\div v+2\div(\zeta v)).
\end{array}
$$
Therefore,   $(\rho,v)$ satisfies \eqref{eq:original Mach limit} (for some suitable $\Pi$) if and only if 
there exists some $Q_2$ so that 
$$\left\{
\begin{array}{l}
\d_t\rho+u\cdot\nabla\rho-\div(\rho\alpha k\nabla\vartheta)=0,\\[1ex]
\rho(\d_t u+u\cdot\nabla u)-\div(\zeta\nabla u)+\rho\alpha kDu\cdot\nabla\vartheta
+\nabla u\cdot\nabla\zeta+\div(\beta v\otimes\nabla\vartheta)+\nabla Q_2=0,\\[1ex]
\div u=0,
\end{array}
\right.
 $$
with
\begin{equation}\label{eq:beta}
\beta:=C_p^{-1}\vartheta^{-1}k+2\zeta'.
\end{equation}
Finally, using  that $\div u=0$ we get (denoting by $B$ a primitive of $\beta$)
$$\begin{array}{lll}
\div(\beta v\otimes\nabla\vartheta)&=&\div(\beta u\otimes\nabla\vartheta)+\div(\beta\alpha k\nabla\vartheta\otimes\nabla\vartheta),\\[1ex]
&=&-\beta \nabla u\cdot\nabla\vartheta +\nabla\div(B(\vartheta)u)+
\div(\beta\alpha k\nabla\vartheta\otimes\nabla\vartheta).
\end{array}
$$
So after multiplying the equation for $u$ by  $\rho^{-1}=\alpha C_p\vartheta$ 
and using the fact that
$$
-\alpha C_p\vartheta\div(\zeta\nabla u)=-\div(\mu\nabla u)+\alpha C_p\zeta D u\cdot\nabla\vartheta
\quad\hbox{with}\quad\mu(\vartheta):=\alpha C_p\vartheta\zeta(\vartheta),
$$ 
we get
\eqref{equation:local} with one new $Q$ and
\begin{equation}\label{eq:A}
\begin{array}{c}
A_1=-\alpha^2C_p\vartheta(\beta k)',\quad
A_2=A_3=-\alpha^2\beta kC_p\vartheta,\\[1.5ex]
A_4=-\rho^{-1}\zeta'+\rho^{-1}\beta=k\alpha+\alpha C_p\vartheta\zeta',\quad
A_5=-\alpha C_p\zeta-\alpha k.
\end{array}
\end{equation}

Motivated by prior works on incompressible
or compressible Navier-Stokes equations with variable density
(see in particular \cite{Danchin00,Danchin01global,Danchin01local,Danchin03}),
we shall use scaling arguments so as to determine the optimal
functional framework for solving the above system.

Here we notice that   if  $(\vartheta,u,\nabla Q)$  is a solution of (\ref{equation:local}),  then 
so does 
\begin{equation}\label{eq:scaling1}
(\vartheta(\ell^2t,\ell x),\ell u(\ell^2t,\ell x),\ell^3\nabla Q(\ell^2t,\ell x))\quad\hbox{for all }\ \ell>0.
\end{equation}
Therefore, critical spaces for the initial data $(\vartheta_0,u_0)$ must be norm invariant
by the transform
\begin{equation}\label{eq:scaling2}
(\vartheta_0,u_0)(x)\rightarrow (\vartheta_0(\ell x),\ell u_0(\ell x))\quad\hbox{for all }\ \ell>0.
\end{equation}

\bigbreak

Let us first consider the easier case where the initial temperature
$\vartheta_0$ is close to a constant (say~$1$ to simplify the presentation). 
Then, setting $\theta=\vartheta-1,$  System \eqref{equation:local} recasts in 
\begin{equation}\label{eq:global}
\left\{\begin{array}{ccc}
\d_t\theta+u\cdot\nabla\theta-\bar\kappa\Delta\theta&=&a(\theta),\\
\d_t u+u\cdot\nabla u-\bar\mu\Delta u+\nabla Q&=&c(\theta,u,\nabla Q),\\
\div u&=&0,
\end{array}\right.
\end{equation}
where $\bar\kappa=\kappa(1),$ $\bar\mu=\mu(1)$ and
\begin{align*}
&a(\theta)=\div((\kappa(1+\theta)-\bar\kappa)\nabla \theta)+f(1+\theta),\\
&c(\theta,u,\nabla Q)=\div((\mu(1+\theta)-\bar\mu)\nabla u)-\theta\nabla Q+h(1+\theta,u).
\end{align*}

Let us notice that the following functional space\footnote{The reader
is referred to Definition \ref{def:Besovhom} for the definition of homogeneous Besov spaces.}:
  $$\left(\! L^\infty(\R^+\!;\dot{B}_{p_1,r_1}^{d/{p_1}})\cap
 L^1(\R^+\!;\dot{B}_{{p_1},r_1}^{d/{p_1}+2})\!\right)\times
\left(\!L^\infty(\R^+\!;\dot{B}_{p_2,r_2}^{d/{p_2}-1})\cap
 L^1(\R^+\!;\dot{B}_{p_2,r_2}^{d/{p_2}+1})\!\right)^d
\times \left(\!L^1(\R^+\!;\dot{B}_{p_3,r_3}^{d/{p_3}-1})\!\right)^d$$
 satisfies the  scaling condition  \eqref{eq:scaling1} for any $1\leq p_1,p_2,p_3,r_1,r_2,r_3\leq\infty.$

However, as a $L^\infty$ control over $\theta$ is needed in order to keep
the ellipticity of the second order operators of the system
and since   $\dot B^{d/{p}}_{p,r}\hookrightarrow L^\infty$
if and only if $r=1,$ we shall  assume that $r_1=1.$
In addition, in many places, having 
$\nabla u\in L^1([0,T];L^\infty)$ will be needed. 
This will induce us to choose $r_2=1$ and $r_3=1$ as well. 
So finally, we plan to solve System \eqref{eq:global} in the space
$$
\dot F_T^{p_1,p_2}:=\left(\tilde C_T(\dot {B}_{p_1,1}^{d/{p_1}})\cap
 L^1_T(\dot {B}_{p_1,1}^{d/{p_1}+2})\right)\times
\left(\tilde C_T(\dot {B}_{p_2,1}^{d/{p_2}-1})\cap
 L^1_T(\dot {B}_{p_2,1}^{d/{p_2}+1})\right)^d
\times\left(  L^1_T(\dot {B}_{p_2,1}^{d/{p_2}-1})\right)^d
$$
where $\tilde C_T(\dot B^\sigma_{p,1})$ is a (large) subspace of 
$\cC([0,T]; \dot B^\sigma_{p,1})$ (see Definition \ref{def:Besov,tilde}). 
\smallbreak
In what follows, we shall denote
\begin{align*}
&\|\theta\|_{\dot X^{p_1}(T)}=\|\theta\|_{\tilde L^\infty_T(\dot B^{d/{p_1}}_{p_1,1})}+\|\theta\|_{L^1_T(\dot B^{d/{p_1}+2}_{p_1,1})},\\
&\|u\|_{\dot Y^{p_2}(T)}=\|u\|_{\tilde L^\infty_T(\dot B^{d/{p_2}-1}_{p_2,1})}+\|u\|_{ L^1_T(\dot B^{d/{p_2}+1}_{p_2,1})},\\
&\|\nabla Q\|_{\dot Z^{p_2}(T)}=\|\nabla Q\|_{L^1_T(\dot B^{d/{p_2}-1}_{p_2,1})},
\end{align*}
and we will drop $T$ in $\dot X^{p_1}(T)$, $\dot Y^{p_2}(T)$, $\dot Z^{p_3}(T)$
 if $T=+\infty$.
\smallbreak

Let us now state our main result for \eqref{eq:global} in the slightly nonhomogeneous case
(that is under a smallness condition for $\theta_0$).

\begin{thm}\label{th:global}
Let  $\theta_0\in \dot{B}^{d/{p_1}}_{p_1,1}$ and  $u_0\in
\dot{B}^{d/{p_2}-1}_{p_2,1}$ with $\div u_0=0.$
Assume that
\begin{equation}\label{cond:global,p}
1\leq p_1<2d, \quad
1\leq p_2<\infty,\quad
p_1\leq 2p_2,\quad
\frac{1}{p_1}+ \frac{1}{p_2}>\frac{1}{d},\quad
\frac{1}{p_2}+\frac{1}{d}\geq\frac{1}{p_1}
\ \hbox{ and }\ \frac{1}{p_1}+\frac{1}{d}\geq\frac{1}{p_2}\cdotp
\end{equation}
There exist two  constants $\tau$ and $K$  depending only on
the coefficients of  System \eqref{eq:global} and on  $d,p_1,p_2,$ 
and satisfying the following properties:
 \begin{itemize}
 \item If 
\begin{equation}\label{cond:local}
 \|\theta_0\|_{\dot B^{d/{p_1}}_{p_1,1}}\leq \tau,
\end{equation}

then there exists $T\in (0,+\infty]$ such that System \eqref{eq:global} has a unique solution $(\theta,u,\nabla Q)$ in $\dot F^{p_1,p_2}_T$, which satisfies
$$
\|\theta\|_{\dot X^{p_1}(T)}\leq K\|\theta_0\|_{\dot B^{d/{p_1}}_{p_1,1}}
\quad\hbox{and}\quad \|u\|_{\dot Y^{p_2}(T)}+\|\nabla Q\|_{\dot Z^{p_2}(T)}\leq K\bigl(\|\theta_0\|_{\dot B^{d/{p_1}}_{p_1,1}}+\|u_0\|_{\dot B^{d/p_2-1}_{p_2,1}}\bigr).
$$
\item If 
\begin{equation}\label{cond:global}
\|\theta_0\|_{\dot B^{d/{p_1}}_{p_1,1}}+\|u_0\|_{\dot B^{d/{p_2}-1}_{p_2,1}}\leq \tau,
\end{equation}
then $T=+\infty$ and the unique global solution satisfies
\begin{equation}\label{estimate:global}
\|\theta\|_{\dot X^{p_1}}+\|u\|_{\dot Y^{p_2}}+\|\nabla Q\|_{\dot Z^{p_2}}\leq 
K\bigl(\|\theta_0\|_{\dot B^{d/{p_1}}_{p_1,1}}+\|u_0\|_{\dot B^{d/{p_2}-1}_{p_2,1}}\bigr).
\end{equation}
\end{itemize}
In addition, the flow map $(\theta_0,u_0)\mapsto (\theta,u,\nabla Q)$ is Lipschitz continuous
from $\dot{B}^{d/{p_1}}_{p_1,1}\times\dot{B}^{d/{p_2}-1}_{p_2,1}$ to $\dot F_T^{p_1,p_2}.$
\end{thm}
\begin{rem} {\rm A similar statement for the nonhomogeneous incompressible
Navier-Stokes equations has been obtained by Abidi-Paicu in \cite{AbidiP07}.
There, the  conditions over $p_1$ and $p_2$ 
(which stem from  the structure of the nonlinearities)
are not exactly the same as ours
for there is no gain of regularity over the density and the right-hand
side of the momentum equation  in \eqref{eq:global} is simpler. 

Let us stress that in the above statement, 
the homogeneous Besov spaces for  the velocity are almost the same as for the standard incompressible
Navier-Stokes equation (except that in this latter case,  one may take \emph{any}
space $\dot B^{d/p_2-1}_{p_2,r}$ with $1\leq p_2<\infty$ and $1\leq r\leq\infty$).  
In particular, here one may take $p_2$ as large as we want hence
the regularity exponent  $d/p_2-1$ may be negative and  our result ensures
that suitably 
oscillating \emph{large} velocities  give rise to a global solution.}
\end{rem}
The important observation for solving \eqref{eq:global} is that all the ``source terms'' (that
is the terms on the right-hand side) are at least quadratic. 
In a suitable functional framework --the one given by our scaling considerations, we thus expect them to be negligible 
if the initial data are small. Hence, appropriate a priori estimates 
for the linearized system pertaining to \eqref{eq:global}
and suitable  product estimates suffice  to control 
the solution for all time if the data are small. This will enable us to prove
the global existence. 
In addition, a classical argument borrowed from the one
that is used in the constant density case will  enable us to consider large
 $u_0.$ 
\bigbreak
 Let us now turn to the fully nonhomogeneous case.
 Then, in order to ensure the ellipticity of the second order operators in the left-hand side of \eqref{equation:local},  we have to assume that $\vartheta_0$ is bounded by below by some positive constant. Proving a priori estimates for the heat or Stokes equations 
 with variable time-dependent rough coefficients will be the key 
 to our local existence statement. Bounding the gradient of the pressure 
 (namely $\nabla Q$)  is the main difficulty. To achieve it, we will have
 to consider the elliptic equation 
 \begin{equation}\label{equation:Q}
\div(\vartheta\nabla Q)=\div L\quad\hbox{with}\quad  
L:=-u\cdot \nabla u+Du\cdot\nabla\mu+h.
\end{equation} 
In the energy framework (that is in Sobolev spaces $H^s$
or in Besov spaces $B^{d/p_2}_{p_2,1}$ with $p_2=2$), this is quite 
standard. At the same time,  if $p_2\not= 2,$ 
estimating $\nabla Q$ in $B^{d/p_2-1}_{p_2,1}$ requires some low order information in $L^{p_2}$ 
over $\nabla Q.$  Thanks to suitable functional embedding, we shall see that
if $p_2\geq2$ then it suffices to bound $\nabla Q$ in $L^2,$
an information which readily stems from the standard $L^2$ elliptic theory. 
As a consequence,  we will have  to restrict our attention 
to a functional framework which ensures that $L$ belongs to $L^2.$
This will induce us to make further assumptions on $p_1$ and 
$p_2$ (compared to (\ref{cond:global,p}) in Theorem \ref{th:global}) so as to ensure
in particular that $\nabla Q$ is in $L^2.$ Consequently,  
 the homogeneous critical framework is no longer appropriate
since some additional control will be required over the low frequencies of the solution.

More precisely, we shall prove the existence of a solution in 
 the following   nonhomogeneous space $F_T^{p_1,p_2}$:
$$
\left(\tilde C_T({B}_{p_1,1}^{d/{p_1}})\cap
 L^1_T({B}_{p_1,1}^{d/{p_1}+2})\right)\times
\left(\tilde C_T({B}_{p_2,1}^{d/{p_2}-1})\cap
L^1_T({B}_{p_2,1}^{d/{p_2}+1})\right)^d
\times\left(L^1_T({B}_{p_2,1}^{d/{p_2}-1}\cap L^2)\right)^d,
$$
which are critical in terms of regularity but more demanding concerning
the behavior at infinity. 
\smallbreak
In what follows,  we denote
\begin{align*}
&\|\theta\|_{X^{p_1}(T)}=\|\theta\|_{\tilde L_T^\infty({B}_{p_1,1}^{d/{p_1}})}
+\|\theta\|_{ L_T^1({B}_{p_1,1}^{d/{p_1}+2})},\\
&\|u\|_{Y^{p_2}(T)}=\|u\|_{\tilde  L_T^\infty({B}_{p_2,1}^{d/{p_2}-1})}
+\|u\|_{L_T^1({B}_{p_2,1}^{d/{p_2}+1})},\\
&\|\nabla Q\|_{Z^{p_2}(T)}=\|\nabla Q\|_{L_T^1({B}_{p_2,1}^{d/{p_2}-1}\cap L^2)}.
\end{align*}
Let us state our main local-in-time existence result for the fully nonhomogeneous case.
\begin{thm}\label{th:local}
 Let $(p_1,p_2)$ satisfy 
  \begin{equation}\label{cond:local,p,r}
  1<p_1\leq 4,\quad 2\leq p_2\leq 4, \quad
   \frac{1}{p_2}+\frac{1}{d}>\frac{1}{p_1}
   \end{equation}
   with in addition 
    \begin{equation*}
  p_1<4  \ \hbox{ if }\  d=2,\quad\hbox{and}\quad
   \frac{1}{p_1}+\frac{1}{d}> \frac{1}{p_2}\ \hbox{ if  }\ d\geq 3.
                   \end{equation*}

For any  initial temperature $\vartheta_0=1+\theta_0$ and velocity field $u_0$ which satisfy
\begin{equation}\label{cond:local initial}\begin{array}{c}
0<m\leq \vartheta_0,\quad\div u_0=0\quad
\hbox{and}\quad \|\theta_0\|_{{B}^{d/{p_1}}_{p_1,1}}+\|u_0\|_{{B}^{d/{p_2}-1}_{p_2,1}}\leq M,
\end{array}\end{equation}
for  some positive constants $m,$ $M,$ 
  there exists a positive time $T$ depending only on $m$, $M,$ $p_1,$ $p_2,$ $d,$   and on the parameters
  of the system such that \eqref{equation:local} has a unique solution
$(\vartheta,u,\nabla Q)$ with $(\theta,u,\nabla Q)$ in $F_T^{p_1,p_2}.$ Furthermore, 
for some constant $C=C(d,p_1,p_2),$ we have 
\begin{equation}\label{estimate:local}
m\leq \vartheta\quad\hbox{and}\quad
\|\theta\|_{X^{p_1}(T)}+\|u\|_{Y^{p_2}(T)}+\|\nabla Q\|_{Z^{p_2}(T)}\leq CM,
\end{equation}
and  the flow map  $(\theta_0,u_0)\mapsto(\theta,u,\nabla Q)$ is Lipschitz continuous.
\end{thm}

\begin{rem} \label{remark:original solution}
The above theorems \ref{th:global} \ref{th:local} and the transformation \eqref{eq:newvelocity}   ensure that the original system \eqref{eq:original Mach limit}
is  well-posed. More precisely, in the case $1\leq p_1=p_2<2d$
 for the initial data $(\vartheta_0,v_0)$  satisfying the third equation,   and $(\vartheta_0-1,v_0)$ in $\dot B^{d/{p_1}}_{p_1,1}\times \dot B^{d/{p_1}-1}_{p_1,1}$
 with \eqref{cond:local},   we  get a  local solution $(\vartheta,v,\nabla \Pi)$ 
 of \eqref{eq:original Mach limit} such that $(\theta,u,\nabla Q)\in \dot F^{p_1,p_1}_T$
 and the solution is global if \eqref{cond:global} holds. 
 Under the same regularity assumptions with in addition $2\leq p_1\leq 4$
 then if  $\vartheta_0$ is just  bounded from below, we get a local solution  
  in $F^{p_1,p_1}_T.$ In the case   
  $p_1\not=p_2,$  a similar result holds true. It is more complicated to state, though. 
  \end{rem}
  
  Let us end this section with a few comments and a short list of open 
questions that we plan to address in the future. 
\begin{itemize}
\item To simplify the presentation, we restricted to the \emph{free} evolution 
of a solution to \eqref{equation:local}. 
As in e.g. \cite{Danchin01local,Danchin03}, our methods enable us 
to treat the case where the fluid is subject to some external body force.
\item We expect similar results for  equations of state
such as those that have been considered by  Alazard in the Appendix of
\cite{Alazard06b} or, more generally, for reacting flows as in \cite{Embid} (as it only
introduces coupling with parabolic equations involving reactants, the scaling of which is the 
same as that of $\vartheta$). We here restricted our analysis to ideal gases for simplicity only. 
\item In the two-dimensional case, unless $k=\eta=0$ and $\zeta$ is a positive constant
(that is for the incompressible Navier-Stokes equations), the question of global existence for large data is widely open. Note however that our derivation of  \eqref{eq:original Mach limit}
highlights the important role of the parameter $\beta$ defined in \eqref{eq:beta}. 
As a matter of fact, it has been discovered very recently by the second author in \cite{Liao}
that global existence holds true in dimensional two (even for large data) if $\beta=0.$ 
\item As for the classical incompressible Euler equations, 
working in a critical functional framework is no longer
relevant in the inviscid case. However,the approach proposed
in \cite{Danchin10euler}  carries out  to our system (see the forthcoming paper \cite{FL}).   
\item Granted with the above results, it is natural 
to study the asymptotics $\varepsilon$ going to $0$    in the above functional framework. 
This would extend some of the results  of Alazard  in \cite{Alazard06}
to the case of rough data.  
\end{itemize}

The rest of the paper unfolds as follows. 
In the next section, we introduce the main tool for the proof --the Littlewood-Paley
decomposition-- and define Besov spaces and some related functional spaces. 
In passing, we state product laws in those spaces and commutator estimates.
In Section \ref{s:small}, we focus on the proof of 
our first well-posedness result
(pertaining to the case where the initial temperature is close to a constant)
whereas our second well-posedness result  is proved in Section \ref{s:large}. 
The proof of a commutator estimate in postponed in Appendix.

\medbreak\noindent{\bf Acknowledgments: } The authors are indebted to the 
  anonymous referee for his relevant remarks on the first 
version of the paper.


\section{Tools}
\setcounter{equation}{0}

Let us first fix some notation. 
\begin{itemize}
\item
Throughout this paper, $C$ represents some ``harmless'' constant, which can be understood from the context. In some places, we shall alternately use  the notation $A\lesssim B$ instead of $A\leq CB$, and $A\approx B$ means $A\lesssim B$ and $B\lesssim A$.
\item
If $p\in[1,+\infty]$ then we  denote by $p'$ the \emph{conjugated} exponent of
$p$ defined by $1/p+1/p'=1.$
\item If $X$ is a Banach space, $T>0$ and $p\in[1,+\infty]$ then
$L_T^p(X)$ stands for the set of Lebesgue measurable functions $f$ from 
$[0,T)$ to $X$ such that $t\mapsto \|f(t)\|_{X}$ belongs to $L^p([0,T)).$
If $T=+\infty,$ then the space is merely denoted by $L^p(X).$
Finally, if $I$ is some interval of $\R$ then the notation $\cC(I;X)$ stands 
for the set of continuous functions from $I$ to $X.$
\item We shall keep the same notation  $X$ to designate vector-fields with components in $X.$ 
\end{itemize}

\subsection{Basic results on Besov spaces}

First of all we recall briefly the definition of the so-called nonhomogeneous
\emph{Littlewood-Paley} 
decomposition: a dyadic partition of unity with respect to the Fourier variable. More precisely, fix a smooth nonincreasing radial function $\chi$, which is supported in the ball $B(0,\frac43)$ and equals to $1$ in a neighborhood of $B(0,1)$. Set $\varphi(\xi)=\chi(\frac{\xi}{2})-\chi(\xi)$, then we have
$$
\chi(\xi)+\sum_{j\geq 0}\varphi(2^{-j}\xi)=1.
$$

Let $\varphi_j(\xi)=\varphi(2^{-j}\xi)$, $h_j=\mathcal{F}^{-1}\varphi_j$, and $\check{h}=\mathcal{F}^{-1}\chi$. The \emph{dyadic blocks} $(\dj)_{j\in \Z}$ are defined by
\begin{align*}
&\dj u=0 \quad\hbox{  if  }\quad j\leq -2,\\
&\Delta_{-1}u=\chi(D)u=\Int_{\R^d}\check{h}(y)u(x-y)\,dy,\\
&\dj u=\varphi_j(D)u=\Int_{\R^d}h_j(y)u(x-y)\,dy \quad\hbox{  if  }\quad j\geq 0,
\end{align*}
and we also introduce the low-frequency cut-off:
$$
S_j u=\sum_{k\leq j-1}\Delta_k u.$$
 Note that $S_ju=\chi(2^{-j}D)u$ if $j\geq0.$
\medbreak
As shown in e.g. \cite{BCD}, the above dyadic decomposition
satisfies 
$$ \Delta_k\dj u\equiv 0\quad \hbox{ if }\ |k-j|\geq 2
\quad\hbox{and}\quad
\Delta_k(S_{j-1}u\dj u)\equiv 0\quad \hbox{ if }\ |k-j|\geq 5.
$$
In addition, for any tempered distribution $u,$
one may write
 $$
u=\sum_{j\in \Z}\dj u,
$$
and, owing to Bersntein's inequalities (see e.g. \cite{BCD}, Chap. 2), 
$$\begin{array}{l}
\|\dj u\|_{L^{p_1}}\lesssim 2^{d(\frac{1}{p_2}-\frac{1}{p_1})}\|\dj u\|_{L^{p_2}}\quad \hbox{ if }p_1\geq p_2,\\[1ex]
\|D^k(\dj u)\|_{L^p}\lesssim 2^{kj}\|\dj u\|_{L^p},\quad\forall j\geq {-1},\\[1ex]
\|D^k(\dj u)\|_{L^p}\approx 2^{kj}\|\dj u\|_{L^p},\quad\forall j\geq {0}.
\end{array}
$$

We can now define the nonhomogeneous Besov spaces $B^s_{p,r}$  as follows:
\begin{defi}\label{def:Besov}
For $s\in \R$, $(p,r)\in [1,+\infty]^2$, and $u\in \mathcal {S}'(\R^d)$, we set
$$
\|u\|_{B^s_{p,r}}=\left(\sum_{j\geq -1}2^{jsr}\|\dj u\|_{L^p}^r\right)^{1/r}\ \hbox{if }\ 
r<\infty,
\quad\hbox{ and }\quad \|u\|_{B^s_{p,\infty}}:=\sup_{j\geq-1}\bigl\{2^{js}\|\dj u\|_{L^p}\bigr\}.
$$
We then define
$$
B^s_{p,r}=B^s_{p,r}(\R^d):=\bigl\{u\in\mathcal{S}'(\R^d),\|u\|_{B^s_{p,r}}<\infty\bigr\}.
$$
\end{defi}

Throughout, we shall use freely the following classical properties for Besov spaces.

\begin{prop}\label{p:Besov}
The following properties hold true:
\begin{enumerate}[(i)]
\item Action of derivatives: $\|\nabla u\|_{B^{s-1}_{p,r}}\lesssim \|u\|_{B^s_{p,r}}$.

 \item Embedding: $B^{s}_{p_1,r_1}\hookrightarrow B^{s-d(\frac{1}{p_1}-\frac{1}{p_2})}_{p_2,r_2}$ if $p_1\leq p_2$, $r_1\leq r_2,$
 and $B^{\frac dp}_{p,1}\hookrightarrow L^\infty$ for all $p\in[1,\infty].$


\item Real interpolation: $(B^{s_1}_{p,r_1},B^{s_2}_{p,r_2})_{\theta,r'}=B^{(1-\theta) s_1+\theta s_2}_{p,r'}$.
    \end{enumerate}
\end{prop}

When dealing with product of functions in Besov spaces, it is often 
convenient to use paradifferential calculus, a tool that has been introduced by J.-M. Bony in \cite{Bony}. Recall that the paraproduct between $u$ and $v$ is defined by
$$
T_uv=\sum_{j}S_{j-1}u\,\dj v,
$$
and that the remainder of $u$ and $v$ is defined by
$$
R(u,v)=\sum_{j}\dj u\,\tilde\dj v\quad\hbox{with}\quad
\tilde\dj v=(\Delta_{j-1}+\Delta_j+\Delta_{j+1})v.
$$
Then we have the following so-called Bony's decomposition for the product between $u$ and $v$:
\begin{equation}\label{eq:bony}
uv=T_uv+R(u,v)+T_vu=T'_uv+T_vu.
\end{equation}
We shall often use the following  estimates in Besov spaces
for the paraproduct and remainder operators.
\begin{prop}\label{p:paraproduct}
Let  $1\leq r,r_1,r_2,p,p_1,p_2\leq\infty$  with 
 $\frac{1}{r}\leq\min\{1,\frac{1}{r_1}+\frac{1}{r_2}\}$
 and  $\frac 1p\leq \frac1{p_1}+\frac1{p_2}\cdotp$
\begin{itemize}
\item 
If $p\leq p_2$ then we have:
\begin{eqnarray}\label{est:product Besov,paraproduct1}
&\|T_u v\|_{B^{s_1\!+\!s_2\!+\!d(\frac{1}{p}-\frac{1}{p_1}\!-\!\frac{1}{p_2})}_{p,r}}\lesssim
\|u\|_{B^{s_1}_{p_1,r_1}}\|v\|_{B^{s_2}_{p_2,r_2}}
\ \textrm{ if  }\ s_1<d\bigl(\frac{1}{p_1}+\frac 1{p_2}-\frac1p\bigr),\\\label{est:product Besov,paraproduct2}
&\|T_u v\|_{B^{s_2}_{p,r}}\lesssim
\|u\|_{L^{p_3}}\|v\|_{B^{s_2}_{p_2,r}}\, \textrm{ if }\, \frac1{p_3}=\frac1p-\frac1{p_2}\cdotp
\end{eqnarray}
\item If $s_1+s_2+d\min\{0,1-\frac{1}{p_1}-\frac{1}{p_2}\}>0$, then
\begin{equation}\label{est:product Besov,remainder}
\|R(u,v)\|_{B^{s_1+s_2+d(\frac{1}{p}-\frac{1}{p_1}-\frac{1}{p_2})}_{p,r}}\lesssim \|u\|_{B^{s_1}_{p_1,r_1}}\|v\|_{B^{s_2}_{p_2,r_2}};
\end{equation}
\item if $s_1+s_2+d\min\{0,1-\frac{1}{p_1}-\frac{1}{p_2}\}=0$ and 
$\frac1{r_1}+\frac1{r_2}\geq1$ then
\begin{equation}\label{est:product Besov,remainder,limite}
\|R(u,v)\|_{B^{s_1+s_2+d(\frac{1}{p}-\frac{1}{p_1}-\frac{1}{p_2})}_{p,\infty}}\lesssim \|u\|_{B^{s_1}_{p_1,r_1}}\|v\|_{B^{s_2}_{p_2,r_2}}.
\end{equation}
\end{itemize}
\end{prop}
\begin{proof}
Most of these results are classical (see e.g. \cite{BCD}).
We just prove  \eqref{est:product Besov,paraproduct1} and \eqref{est:product Besov,paraproduct2}.
which is a slight generalization of  Prop. 2.3 in  \cite{AbidiP07}. 
We write that 
$$
T_uv=\sum_{j\geq1} T_j(u,v)\quad\hbox{with }\ T_j(u,v)=S_{j-1}u\dj v.
$$
Since  $\Delta_{j'}(S_{j-1}u\dj v)\equiv0$ for $|j'-j|>4,$
 it suffices to show that, for some sequence $(c_j)_{j\in\N}$
such that $\|(c_j)\|_{\ell^r}=1,$ we have
$$
\begin{array}{lll}
\|T_j(u,v)\|_{L^p} \lesssim c_j2^{-js}
\|u\|_{B^{s_1}_{p_1,r_1}}\|v\|_{B^{s_2}_{p_2,r_2}}&\hbox{if}&
s_1<d/p_1+d/p_2-d/p,\\[1ex]
\|T_j(u,v)\|_{L^p} \lesssim c_j2^{-js}
\|u\|_{L^{p_3}}\|v\|_{B^{s_2}_{p_2,r_2}}&\hbox{if}&
s_1=d/p_1+d/p_2-d/p,
\end{array}
$$
with  $s=s_1+s_2+\frac d{p}-\frac d{p_1}-\frac d{p_2}.$
\medbreak
According to H\"older's inequality, we have
\begin{equation}\label{eq:limite}
\|T_j(u,v)\|_{L^p}\leq \|S_{j-1}u\|_{L^{p_3}}\|\dj v\|_{L^{p_2}}
\quad\hbox{with }\ \frac1{p_3}=\frac1p-\frac1{p_2}\cdotp
\end{equation}
Hence, using the definition of $S_{j-1}$ and Bernstein's inequality (here
we notice that  $p_1\leq p_3,$ a consequence of $\frac1{p}\leq\frac1{p_1}+\frac1{p_2}$):
$$
\|T_j(u,v)\|_{L^p}\lesssim \sum_{j'\leq j-2} 2^{j'(\frac d{p_1}-\frac d{p_3})}
 \|\Delta_{j'}u\|_{L^{p_1}}\|\dj v\|_{L^{p_2}},
$$
whence
$$
2^{js}\|T_j(u,v)\|_{L^p}\leq  \sum_{j'\leq j-2} 2^{(j-j')(s_1+\frac d{p_3}-\frac d{p_1})}
\bigl(2^{j's_1}\|\Delta_{j'}u\|_{L^{p_1}}\bigr)
\bigl(2^{js_2}\|\dj v\|_{L^{p_2}}\bigr).
$$
Therefore,  if $s_1+d/p_3-d/p_1<0$ then
the result stems from convolution and H\"older inequalities for series.
In the case where $s_1+d/p_3-d/p_1=0,$ we just have to use
that  $\|S_{j-1}u\|_{L^{p_3}}\leq C\|u\|_{L^{p_3}}$ in \eqref{eq:limite}. 
\smallbreak
The proof of \eqref{est:product Besov,remainder}
goes from similar arguments and is thus left to the reader (see also \cite{AbidiP07}).
\end{proof}

{}From the above proposition and \eqref{eq:bony}, one may deduce a number of
estimates  in Besov spaces for the product of two functions.
We shall use the following result:
\begin{prop}\label{prop:product Besov}
The following estimates hold true:
\begin{enumerate}[(i)]
\item $\|uv\|_{B^s_{p,r}}\lesssim \|u\|_{L^\infty}\|v\|_{B^s_{p,r}}+\|u\|_{B^s_{p,r}}\|v\|_{L^\infty}$ if $s>0$.
\item If $s_1<\frac{d}{p_1}$, $s_2<d\min\{\frac{1}{p_1},\frac{1}{p_2}\}$, 
$s_1+s_2+d\min\{0,1-\frac{1}{p_1}-\frac{1}{p_2}\}>0$ and $\frac1r\leq
\min\bigl\{1,\frac1{r_1}+\frac1{r_2}\bigr\}$  then
\begin{equation}\label{est:product Besov1}
\|uv\|_{B^{s_1+s_2-\frac{d}{p_1}}_{p_2,r}}\lesssim \|u\|_{B^{s_1}_{p_1,r_1}}\|v\|_{B^{s_2}_{p_2,r_2}}.
\end{equation}
\item We also have the following  limit cases:
\begin{itemize}
\item if $s_1=d/p_1,$ $s_2<\min(d/p_1,d/p_2)$ 
and $s_2+d\min(1/p_1,1/p'_2)>0$  then 
\begin{equation}\label{est:product Besov2}
\|uv\|_{B^{s_2}_{p_2,r}}\lesssim \|u\|_{B^{\frac d{p_1}}_{p_1,\infty}\cap L^\infty}\|v\|_{B^{s_2}_{p_2,r}};
\end{equation}
\item if $s_2=\min(d/p_1,d/p_2),$ $s_1<d/p_1$ and $s_1+s_2+d\min\{0,1-\frac{1}{p_1}-\frac{1}{p_2}\}>0$  then
$$\|uv\|_{B^{s_1+s_2-\frac d{p_1}}_{p_2,r}}\lesssim \|u\|_{B^{s_1}_{p_1,r}}\|v\|_{B^{s_2}_{p_2,1}};
$$
\item  if $1/r_1+1/r_2\geq1,$ $s_1<\frac d{p_1},$ $s_2<d\min\{\frac1{p_1},\frac1{p_2}\}$
and $s_1+s_2+d\min(0,1-\frac1{p_1}-\frac1{p_2})=0$ then
$$\|uv\|_{B^{s_1+s_2-\frac d{p_1}}_{p_2,\infty}}\lesssim \|u\|_{B^{s_1}_{p_1,r_1}}\|v\|_{B^{s_2}_{p_2,r_2}}.
$$
\end{itemize}
\end{enumerate}
\end{prop}

The following commutator estimates  (see the proof in Appendix) will be also needed:
\begin{prop}\label{prop:commutator Besov}
Let $\frac{1}{r}=\min\{1,\frac{1}{r_1}+\frac{1}{r_2}\}$ and
$(s,\nu)\in \R\times\R$ satisfying 
\begin{equation}\label{eq:condcom}
-d\min\Bigl\{\frac{1}{p'_1},\frac{1}{p_2}\Bigr\}<s<\nu+d\min\Bigl\{\frac{1}{p_1},\frac{1}{p_2}\Bigr\}\quad\hbox{and}\quad
-d\min\Bigl\{\frac{1}{p_1},\frac{1}{p_2}\Bigr\}<\nu<1.
\end{equation}
For $j\geq-1,$ denote $R_j(u,v):=[u,\dj]v.$ 
We have
\begin{equation}\label{est:com1}
\|\left( 2^{js}\|R_j(u,v)\|_{L^{p_1}} \right)_{j\geq -1}\|_{\ell^r}\lesssim \|\nabla u\|_{B^{\frac{d}{p_2}+\nu-1}_{p_2,r_2}}\|v\|_{B^{s-\nu}_{p_1,r_1}}.
\end{equation}

The following limit cases also hold true:
\begin{itemize}
\item if $s=\nu+d\min\{\frac{1}{p_1},\frac{1}{p_2}\},$ $r_1=1$ and $r_2=r$
 then we have
\begin{equation}\label{est:com2}
\|\left( 2^{js}\|R_j(u,v)\|_{L^{p_1}} \right)_{j\geq -1}\|_{\ell^r}\lesssim 
 \|\nabla u\|_{B^{\frac{d}{p_2}+\nu-1}_{p_2,r}}\|v\|_{B^{s-\nu}_{p_1,1}};
 \end{equation}
 \item if $\nu=1,$ $r_1=r$ and $r_2=\infty$  then we have 
 \begin{equation}\label{est:com3}
 \|\left( 2^{js}\|R_j(u,v)\|_{L^{p_1}} \right)_{j\geq -1}\|_{\ell^r}\lesssim 
 \|\nabla u\|_{B^{\frac{d}{p_2}}_{p_2,\infty}\cap L^\infty}\|v\|_{B^{s-1}_{p_1,r}}.
\end{equation}
\end{itemize}
Finally, if in addition to \eqref{eq:condcom}, we have $\nu>1-d\min(1/p_1,1/p_2)$ then 
\begin{equation}\label{est:com4}
\|\bigl( 2^{j(s-1)}\|\d_kR_j(u,v)\|_{L^{p_1}} \bigr)_{j\geq -1}\|_{\ell^r}\lesssim \|\nabla u\|_{B^{\frac{d}{p_2}+\nu-1}_{p_2,r_2}}\|v\|_{B^{s-\nu}_{p_1,r_1}}\quad\hbox{for all }\ k\in\{1,\cdots,d\},
\end{equation}
with the above changes in the limit cases.  
\end{prop}

We shall also use  the following result for the action of smooth functions (see  e.g. \cite{BCD}):
\begin{prop}\label{prop:action Besov} Let $(p,r)\in[1,+\infty]^2$ and $s>0.$
Let $f$ be a smooth function from  $\R$ to $\R.$ 
\begin{itemize}
\item If $f(0)=0$ then for all $u\in B^{s}_{p,r}\cap L^\infty$  we have
\begin{equation}\label{est:action Besov}
\|f\circ u\|_{B^{s}_{p,r}}\leq C(f',\|u\|_{L^\infty})\|u\|_{B^{s}_{p,r}}.
\end{equation}
\item If $f'(0)=0$ then for all $u$ and $v$ in $B^s_{p,r}\cap L^\infty,$ we have
\begin{equation}\label{est:action Besovdiff}
\|f\circ v-f\circ u\|_{B^{s}_{p,r}}\leq C(f'',\|u\|_{L^\infty\cap B^s_{p,r}},\|v\|_{L^\infty\cap B^s_{p,r}})\|v-u\|_{B^{s}_{p,r}}.
\end{equation}
\end{itemize}
\end{prop}

When solving evolutionary PDEs, it is natural to use 
spaces of type $L^\rho_T(X)=L^\rho(0,T;X)$ with $X$ denoting some Banach space. In our case, $X$ will be a  Besov space so that we will 
have to localize the equations through Littlewood-Paley decomposition. This will provide us 
with  estimates of the Lebesgue norm of  each dyadic block \emph{before}
performing integration in time. This  leads to the following  definition:
\begin{defi}\label{def:Besov,tilde}
For  $s\in \R$, $(\rho,p,r)\in [1,+\infty]^3$ and $T\in [0,+\infty]$, we set
$$
\|u\|_{\tilde L^\rho_T(B^s_{p,r})}=\left(\sum_{j\geq -1}2^{rjs}\left(\Int^T_0\|\dj u(t)\|_{L^p}^\rho\, dt\right)^{\frac{r}{\rho}}\right)^{\frac{1}{r}},
$$
with the usual change if $r=+\infty$ or $\rho=+\infty.$

We also set $\tilde C_T(B^s_{p,r})=\tilde L_T^\infty(B^s_{p,r})\cap \cC([0,T];B^s_{p,r}).$
\end{defi}

Let us remark that, by virtue of Minkowski's inequality, we have
\begin{align*}
&\|u\|_{\tilde L^\rho_T(B^s_{p,r})}\leq \|u\|_{L^\rho_T(B^s_{p,r})}\hbox{ if }\rho\leq r,\\
&\|u\|_{L^\rho_T(B^s_{p,r})}\leq \|u\|_{\tilde L^\rho_T(B^s_{p,r})}\hbox{ if }\rho\geq r,
\end{align*}
and hence in particular $\|u\|_{\tilde L^1_T(B^s_{p,1})}= \|u\|_{L^1_T(B^s_{p,1})}$ holds.
\smallbreak
Let $\theta\in [0,1]$, $\frac{1}{\rho}=\frac{\theta}{\rho_1}+\frac{1-\theta}{\rho_2}$, and $s=\theta s_1+(1-\theta)s_2$, then the following interpolation inequality holds true:
$$
\|u\|_{\tilde L^\rho_T(B^s_{p,r})}\leq \|u\|_{\tilde L^{\rho_1}_T(B^{s_1}_{p,r})}^{\theta}
\|u\|_{\tilde L^{\rho_2}_T(B^{s_2}_{p,r})}^{1-\theta}.
$$

In this framework, one may get product or composition estimates similar to those
that have been stated above. The general rule is that the Lebesgue exponents pertaining 
to the time variable behave according to H\"older's inequality. 
For instance, one has:
\begin{equation}\label{eq:timeBesov1}
\|uv\|_{\tilde L^\rho_T(B^s_{p,r})}\lesssim \|u\|_{L^{\rho_1}_T(L^\infty)}\|v\|_{\tilde L^{\rho_4}_T(B^s_{p,r})}
+\|u\|_{\tilde L^{\rho_2}_T(B^s_{p,r})}\|v\|_{L^{\rho_3}_T(L^\infty)},
\end{equation}
whenever $s>0$, 
$\frac{1}{\rho}=\frac{1}{\rho_1}+\frac{1}{\rho_4}=\frac{1}{\rho_2}+\frac{1}{\rho_3},$
and
\begin{equation}\label{eq:timeBesov2}\|uv\|_{\tilde L^\rho_T(B^{s_1+s_2-\frac{d}{p}}_{p,r})}\lesssim \|u\|_{\tilde L^{\rho_1}_T(B^{s_1}_{p,r})}
\|v\|_{\tilde L^{\rho_2}_T(B^{s_2}_{p,\infty})},\end{equation}
 if $s_1+s_2+d\min\{0,1-\frac{2}{p}\}>0$, $s_1,s_2<\frac{d}{p}$, $\frac{1}{\rho}=\frac{1}{\rho_1}+\frac{1}{\rho_2}\cdotp$
\medbreak

As pointed out in the introduction,  scaling invariant spaces have to be homogeneous.
As a consequence, the optimal framework for proving our first well-posedness result
(namely Theorem \ref{th:global}) turns to be 
\emph{homogeneous} Besov spaces. 
For completeness, we here define those spaces. 
We first need to introduce \emph{homogeneous} dyadic blocks
$$
\ddj u=\Int_{\R^d}h_j(y)u(x-y)dy,\quad \forall j\in\Z
$$
and the homogeneous low-frequency truncation operator 
\begin{equation}\label{eq:dotSj}
\dot S_j:=\chi(2^{-j}D),\quad\forall j\in\Z.
\end{equation}
We then define homogeneous semi-norms:
$$
\|u\|_{\dot B^s_{p,r}}=\bigl\|(2^{js}\|\ddj u\|_{L^p})_{j\in\Z}\bigr\|_{\ell^r}.
$$
Note that, for $u\in\cS'(\R^d),$ the equality
$$
u=\sum_{j\in\Z}\ddj u
$$
holds true modulo polynomials only. 
Hence, the functional spaces related to the above semi-norm cannot be defined
without care. Following \cite{BCD}, we shall define homogeneous Besov spaces as follows:
\begin{defi}\label{def:Besovhom}
The homogeneous Besov space $\dot B^s_{p,r}$ is the set of 
tempered distributions $u$ such that 
$$\|u\|_{\dot B^s_{p,r}}<\infty\quad\hbox{and}\quad
\lim_{j\rightarrow-\infty}\|\dot S_j u\|_{L^\infty}=0.
$$
\end{defi}
The above definition ensures that   $\dot B^s_{p,r}(\R^d)$
is a Banach space provided that 
\begin{equation}\label{eq:condition}
s<d/p\quad\hbox{or}\quad s\leq d/p\ \hbox{ if }\ r=1.
\end{equation}

All the above estimates remain true in 
homogeneous spaces. In addition, if $u=\sum_{j\in\Z} \ddj u$
and $\|u\|_{\dot B^s_{p,r}}$ is finite for some $(s,p,r)$ satisfying \eqref{eq:condition}
then $u$ belongs to $\dot B^s_{p,r}(\R^d),$ owing to the aforementioned 
Bernstein's inequalities.
This fact will be used repeatedly.


\section{The proof of Theorem \ref{th:global}}\label{s:small}\setcounter{equation}{0}

This section is devoted to  the  well-posedness  issue for System
\eqref{eq:global} in the slightly nonhomogeneous case. 
The proof strongly relies on a priori estimates 
for the linearized equations about $0$ which will be recalled
in the first part of this section.
The proof of existence and uniqueness will be carried out
in the second part.

\subsection{The linearized equations}

In the case of a given velocity field $w,$ 
the linearized temperature equation about $0$ reads 
\begin{equation}\label{eq:TD}
\left\{\begin{array}{lll}
\partial_t\theta+w\cdot\nabla \theta-\bar\kappa\Delta \theta=f,\\
\theta_{|t=0}=\theta_0.
\end{array}\right.
\end{equation}

Obviously, the convection term $w\cdot\nabla\theta$ is of lower order
so that it may be included in the ``source terms'' if it is only
a matter of solving \eqref{eq:global}. 
However, considering the above convection-diffusion equation \eqref{eq:TD}
rather than the standard heat equation will enable us to get more accurate
estimates. The same remark  holds for the linearized momentum
equation \eqref{eq:Stokes}: 
\begin{equation}\label{eq:Stokes}
\left\{\begin{array}{lll}
\partial_tu+w\cdot\nabla u-\bar\mu\Delta u+\nabla Q=h,\\
\div u=0,\\
u_{|t=0}=u_0.
\end{array}\right.
\end{equation}

The reader is referred to \cite{Danchin07} for the proof 
of the following two results.

\begin{prop}\label{th:TDnu} 
Let $1\leq p\leq p_1\leq\infty$ and $1\leq  r\leq\infty.$
Let $s\in\R$ satisfy
\begin{equation}\label{eq:indicehomo}
\left\{\begin{array}{l}
\ds s<1+\frac d{p_1},\quad\hbox{\rm or}
\quad s\leq 1+\frac d{p_1}\
  \hbox{ if }\ r=1,\\[1.5ex]
\displaystyle s>-d\min\Bigl\{\frac1{p_1},\frac1{p'}\Bigr\},\ \ \hbox{\rm or}\ \
s>-1-d\min\Bigl\{\frac1{p_1},\frac1{p'}\Bigr\}\ \; \hbox{\rm if}\ \;
\div w=0.
\end{array}
\right.
\end{equation}
There exists a constant  $C$ depending only on $d,$ $r,$ $s$ and
$s-1-\frac{d}{p_1}$ such that  for any smooth solution $\theta$ of \eqref{eq:TD}
with $\bar\kappa\geq0,$ and $\rho\in[1,\infty],$  we
have  the following a priori estimate:
$$
\bar\kappa^{\frac1\rho}\|\theta\|_{\tilde L_T^\rho(\dot B^{s+\frac2\rho}_{p,r})}
 \leq  e^{CW_{p_1}(T)}\biggl(\|\theta_0\|_{\dot B^s_{p,r}}+
\|f\|_{\tilde L_T^{1}(\dot
B^{s}_{p,r})}\biggr)$$
with $\quad\displaystyle\left\{\begin{array}{lll}
 W_{p_1}(T) =\Int_0^T
\|\nabla w(t)\|_{\dot B^{\frac d{p_1}}_{p_1,\infty}\cap L^\infty}\,dt&\hbox{ if}& 
s<\frac d{p_1}+1,\\[1ex]
 W_{p_1}(T) =  \Int_0^T\|\nabla w(t)\|_{\dot B^{\frac d{p_1}}_{p_1,1}}\,dt&\hbox{ if}& 
s=\frac d{p_1}+1.\end{array}\right.
$
\end{prop}

\begin{prop}\label{th:Stokes} 
Let $p,$ $p_1,$ $r,$ $s$ and $W_{p_1}$ be as in Proposition \ref{th:TDnu}. 
There exists a constant  $C$ depending only on $d,$ $r,$ $s$ and
$s-1-\frac{d}{p_1}$ such that  for any smooth solution $(u,\nabla Q)$ of \eqref{eq:Stokes}
with $\bar\mu\geq0,$ and $\rho\in[1,\infty],$  we
have  the following a priori estimate:
$$\displaylines{
\bar\mu^{\frac1\rho}\|u\|_{\tilde L_T^\rho(\dot B^{s+\frac2\rho}_{p,r})}
 \leq  e^{CW_{p_1}(T)}\biggl(\|u_0\|_{\dot B^s_{p,r}}+
\|\cP h\|_{\tilde L_T^{1}(\dot
B^{s}_{p,r})}\biggr),\cr
\|\nabla Q-\cQ h\|_{\tilde L_T^{1}(\dot B^{s}_{p,r})}
\leq C\biggl(e^{CW_{p_1}(T)}-1\biggr)
\biggl(\|u_0\|_{\dot B^s_{p,r}}+
\|\cP h\|_{\tilde L_T^{1}(\dot
B^{s}_{p,r})}\biggr)\cdotp}
$$
\end{prop}
Above, $\cP$ and $\cQ$ stand for the orthogonal projectors
over divergence-free and potential vector-fields, respectively.


\subsection{The  well-posedness issue in the slightly nonhomogeneous case}

For proving existence, we will follow a standard procedure, 
first we  construct a sequence of  approximate solutions,  second, we prove 
uniform bounds for them, and finally we  show the convergence to some solution 
of the system. In the case of \emph{large} initial velocity,  we will have 
to split the constructed velocity into the free solution of the Stokes 
system with  initial data $u_0,$   
and the discrepancy to this free velocity. 
Stability estimates and uniqueness will be obtained afterward
by the same argument as the convergence of the sequence.

\subsubsection*{Step 1. Approximate solutions}
Solving System \eqref{eq:global}  will be based on an iterative scheme: first
 we set $(\theta^0,u^0,\nabla Q^0)\equiv0$ then, 
 once $(\theta^n,u^n,\nabla Q^n)$ has been defined over $\R^+\times\R^d,$
 we define $(\theta^{n+1},u^{n+1},\nabla Q^{n+1})$ as
 the solution to the following linear system\footnote{Note that the existence of solution 
for this system may be deduced from the case with no convection. Indeed, 
considering the convection terms as source terms, it is not difficult to construct 
an iterative scheme the convergence of which is based on  the estimates of the previous subsection.} : 
\begin{equation}\label{eq:globaln}
\left\{\begin{array}{ccc}
\d_t\theta^{n+1}+u^n\cdot\nabla\theta^{n+1}-\bar\kappa\Delta\theta^{n+1}&=&a^n,\\
\d_t u^{n+1}+u^n\cdot\nabla u^{n+1}-\bar\mu\Delta u^{n+1}+\nabla Q^{n+1}&=&c^n,\\
\div u^{n+1}&=&0,\\
(\theta^{n+1},u^{n+1})_{|t=0}&=&(\dot S_{n+1}\theta_0,\dot S_{n+1}u_0),
\end{array}\right.
\end{equation}
with $\dot S_{n+1}$ defined in \eqref{eq:dotSj} and, denoting $\vartheta^n:=1+\theta^n,$
$$
\displaylines{
a^n:=a^n(\theta^n)=\div((\kappa(\vartheta^n)-\bar\kappa)\nabla \theta^n)
-\kappa'(\vartheta^n)|\nabla\theta^n|^2,\hfill\cr
c^n:=c^n(\theta^n,u^n,\nabla Q^n)=\div((\mu(\vartheta^n)-\bar\mu)\nabla u^n)-\theta^n\nabla Q^n
+A_1^n|\nabla\theta^n|^2\nabla\theta^n
\hfill\cr
\hfill+A_2^n\Delta\theta^n\nabla\theta^n+A_3^n\nabla^2\theta^n\!\cdot\!\nabla\theta^n
+A_4^n\nabla u^n\cdot\nabla \theta^n
+A_5^nDu^n\cdot\nabla\theta^n.}
$$
Above, it is understood that $A_i^n:=A_i(\vartheta^n)$ with  $A_i$ defined by \eqref{eq:A}.

\subsubsection*{Step 2.  Uniform bounds}

In order to bound $(\theta^{n+1},u^{n+1},\nabla Q^{n+1}),$ one may
take advantage of Proposition  \ref{th:TDnu} with   $s=d/p_1$ 
 and  Lebesgue exponents $(p_1,p_2)$  (here comes the assumption
 that  $1/p_1\leq 1/d+1/p_2$),  and of
 Proposition \ref{th:Stokes} with   $s=d/p_2-1$ and exponents $(p_2,p_2)$.
 Concerning $\theta^{n+1},$  if $p_2\leq p_1$ then we  use
the embedding $\dot B^{{\frac d{p_2}+1}}_{p_2,1}\hookrightarrow 
\dot B^{\frac d{p_1}+1}_{p_1,1}.$
We eventually get 
\begin{eqnarray}\label{eq:tn}
\|\theta^{n+1}\|_{\dot X^{p_1}(t)}\lesssim
e^{\|u^n\|_{L^{1}_t(\dot B^{d/{p_2}+1}_{p_2,1})}}
\Bigl(\|\dot S_{n+1}\theta_0\|_{\dot B^{d/{p_1}}_{p_1,1}}+\|a^n\|_{L^1_t(\dot B^{d/{p_1}}_{p_1,1})}\Bigr),\\\label{eq:un}
\|u^{n+1}\|_{\dot Y^{p_2}(t)}+\|\nabla Q^{n+1}\|_{\dot Z^{p_2}(t)}
\lesssim e^{\|u^n\|_{L^{1}_t(\dot B^{d/{p_2}+1}_{p_2,1})}}
\Bigl(\|\dot S_{n+1} u_0\|_{\dot B^{d/{p_2}-1}_{p_2,1}}+\|c^n\|_{L^1_t(\dot B^{d/{p_2}-1}_{p_2,1})}\Bigr).
\end{eqnarray}
Let us now bound $a^n$ and $c^n.$
Using Propositions \ref{prop:product Besov} and \ref{prop:action Besov}, we easily get
\begin{equation}\label{eq:an}
\|a^n\|_{\dot B^{\frac d{p_1}}_{p_1,1}}\lesssim \bigl(1+\|\theta^n\|_{\dot B^{\frac d{p_1}}_{p_1,1}}\bigr) \bigl(\|\theta^n\|_{\dot B^{\frac d{p_1}}_{p_1,1}}
\|\theta^n\|_{\dot B^{\frac d{p_1}+2}_{p_1,1}}
+\|\nabla\theta^n\|_{\dot B^{\frac d{p_1}}_{p_1,1}}^2\bigr).
\end{equation}
As regards $c^n,$ it is mostly a matter of bounding
the following terms in $L_t^1(\dot B^{\frac d{p_2}-1}_{p_2,1})$
(keeping in mind that $1/p_2\leq1/d+1/p_1$):
$$
\nabla^2\theta^n\cdot\nabla\theta^n,\quad|\nabla\theta^n|^2\nabla\theta^n,\quad
\div(\theta^n\nabla u^n),\quad
\nabla\theta^n\otimes\nabla u^n\ \hbox{ and }\ 
\theta^n\nabla Q^n.
$$
Indeed, on any interval $[0,T],$ 
taking  the   $\vartheta^n$ dependency of the coefficients  into account
will only multiply  the estimates by  some continuous function of 
$\|\theta^n\|_{L_T^\infty(\dot B^{\frac d{p_1}}_{p_1,1})}.$
In what follows, this function will be denoted by~$C_{\theta^n}.$
\smallbreak

Now, if $p_1<2d$ then  Proposition \ref{prop:product Besov}
ensures that the usual product maps $\dot B^{\frac d{p_1}-1}_{p_1,1}\times
\dot B^{\frac d{p_1}}_{p_1,1}$ in $\dot B^{\frac d{p_1}-1}_{p_1,1}.$ Therefore,
if $p_1\leq p_2,$ then functional embedding implies that
\begin{equation}\label{}
\|\nabla^2\theta^n\cdot\nabla\theta^n\|_{\dot B^{\frac d{p_2}-1}_{p_2,1}}
\lesssim 
\|\nabla^2\theta^n\|_{\dot B^{\frac d{p_1}-1}_{p_1,1}}\|\nabla\theta^n\|_{\dot B^{\frac d{p_1}}_{p_1,1}}.
\end{equation}

To deal with the more complicated case where $p_1>p_2,$
we use the following Bony's decomposition:
$$
\nabla^2\theta^n\cdot\nabla\theta^n=T_{\nabla^2\theta^n}\nabla\theta^n
+R(\nabla^2\theta^n,\nabla\theta^n)+T_{\nabla\theta^n}{\nabla^2\theta^n}.
$$

Finally, Proposition \ref{p:paraproduct} enables to conclude that under conditions
$$
p_1<2d,\quad p_1\leq 2p_2\quad\hbox{and}\quad
\frac1{p_2}\leq\frac1{p_1}+\frac1d,
$$
we have
\begin{equation}\label{}
\|\nabla^2\theta^n\cdot\nabla\theta^n\|_{\dot B^{\frac d{p_2}-1}_{p_2,1}}
\lesssim 
\|\nabla^2\theta^n\|_{\dot B^{\frac d{p_1}-1}_{p_1,1}}\|\nabla\theta^n\|_{\dot B^{\frac d{p_1}}_{p_1,1}}+\|\nabla^2\theta^n\|_{\dot B^{\frac d{p_1}}_{p_1,1}}\|\nabla\theta^n\|_{\dot B^{\frac d{p_1}-1}_{p_1,1}}.
\end{equation}
Bounding $|\nabla\theta^n|^2\nabla\theta^n$ stems from similar 
arguments. Under the above conditions, it is found that
\begin{equation}\label{}
\||\nabla\theta^n|^2\nabla\theta^n\|_{\dot B^{\frac d{p_2}-1}_{p_2,1}}
\lesssim \|\nabla\theta^n\|_{\dot B^{\frac d{p_1}-1}_{p_1,1}}
\|\nabla\theta^n\|_{\dot B^{\frac d{p_1}}_{p_1,1}}^2.
\end{equation}
We also easily get
\begin{equation}\label{}
\|\div(\theta^n\nabla u^n)\|_{\dot B^{\frac d{p_2}-1}_{p_2,1}}
\lesssim \|\theta^n\|_{\dot B^{\frac d{p_1}}_{p_1,1}}
\|\nabla u^n\|_{\dot {B}^{\frac d{p_2}}_{p_2,1}}+
 \|\theta^n\|_{\dot B^{\frac d{p_1}+1}_{p_1,1}}
\|\nabla u^n\|_{\dot {B}^{\frac d{p_2}-1}_{p_2,1}}.
\end{equation}

Finally, according to Inequality \eqref{est:product Besov2}, 
$$
\begin{array}{lll}
\|\nabla u^n\otimes\nabla\theta^n\|_{\dot B^{\frac d{p_2}-1}_{p_2,1}}
&\lesssim& \|\nabla\theta^n\|_{\dot B^{\frac d{p_1}}_{p_1,1}}
\|\nabla u^n\|_{\dot B^{\frac d{p_2}-1}_{p_2,1}},\\[1ex]
\|\theta^n\nabla Q^n\|_{\dot B^{\frac d{p_2}-1}_{p_2,1}}&\lesssim&  \|\theta^n\|_{\dot B^{\frac d{p_1}}_{p_1,1}}
\|\nabla Q^n\|_{\dot B^{\frac d{p_2}-1}_{p_2,1}},
\end{array}
$$
provided that 
\begin{equation}
\frac1{p_1}+\frac1{p_2}>\frac1d\quad\hbox{and}\quad
\frac1{p_1}+\frac1d\geq\frac1{p_2}\cdotp
\end{equation}

So, plugging all the above inequalities  in \eqref{eq:tn},\eqref{eq:un}  finally implies  that
\begin{eqnarray}\label{eq:tn1}
\|\theta^{n+1}\|_{\dot X^{p_1}(t)}
\lesssim e^{C\|u^n\|_{L^1_t(\dot B^{\frac d{p_2}+1}_{p_2,1})}}\Bigl(\|\theta_0\|_{\dot B^{\frac d{p_1}}_{p_1,1}}
+C_{\theta^{n}}\|\theta^{n}\|_{\dot X^{p_1}(t)}^2\Bigr),\\\label{eq:un1}
 \|u^{n+1}\|_{\dot Y^{p_2}(t)}+\|\nabla Q^{n+1}\|_{\dot Z^{p_2}(t)}\lesssim
 e^{C\|u^n\|_{L^1_t(\dot B^{\frac d{p_2}+1}_{p_2,1})}}\Bigl(\|u_0\|_{\dot B^{\frac d{p_2}-1}_{p_2,1}}\qquad\qquad
 \nonumber\\\qquad\qquad+C_{\theta^n}\|\theta^n\|_{\dot X^{p_1}(t)}
 \bigl(\|\theta^n\|_{\dot X^{p_1}(t)}
 + \|u^{n}\|_{\dot Y^{p_2}(t)}
  + \|\nabla Q^{n}\|_{\dot Z^{p_2}(t)}\bigr)\Bigr).
  \end{eqnarray}

{}Note that the right-hand sides involves
only initial data and at least quadratic combinations of 
the norms of $(\theta^n,u^n,\nabla Q^n).$ {}From a standard induction argument, it is thus easy to find some small 
constant $\tau$ such that if
\begin{equation}\label{}
\|\theta_0\|_{\dot B^{\frac d{p_1}}_{p_1,1}}
+\|u_0\|_{\dot B^{\frac d{p_2}-1}_{p_2,1}}\leq\tau
\end{equation}
then, for all $n\in\N$ and $t\in\R^+,$ we have for some $K>0$ depending
only on the parameters of the system and on $d,p_1,p_2,$
\begin{equation}\label{eq:unifbound1}
\|\theta^{n}\|_{\dot X^{p_1}(t)}
+ \|u^{n}\|_{\dot Y^{p_2}(t)}+\|\nabla Q^{n}\|_{\dot Z^{p_2}(t)}
\leq K\bigl(\|\theta_0\|_{\dot B^{\frac d{p_1}}_{p_1,1}}+\|u_0\|_{\dot B^{\frac d{p_2}-1}_{p_2,1}}\bigr).
\end{equation}
This completes the proof of uniform estimates in the case where both $\theta_0$ and 
$u_0$ are small. 
\medbreak
Let us now concentrate on the case where only $\theta_0$ is small.
Assuming that $T$ has been chosen so that 
\begin{equation}\label{eq:smallT}
\exp\biggl(C\int_0^T\|u^n\|_{\dot B^{\frac d{p_2}+1}_{p_2,1}}\,dt\biggr)\leq2,
\end{equation}
and that $\theta_0$ is small enough, 
Inequality \eqref{eq:tn1} still implies that 
\begin{equation}\label{eq:tn2}
\|\theta^{n+1}\|_{\dot X^{p_1}(T)}\leq K\|\theta_0\|_{\dot B^{\frac d{p_1}}_{p_1,1}}
\end{equation}
if \eqref{cond:local} is satisfied and if $\theta^n$ also satisfies \eqref{eq:tn2}. 
\smallbreak
However, if $u_0$ is large then Inequality \eqref{eq:un1} is not enough
to bound $u^{n+1}.$
Therefore  we introduce the ``free'' solution $u_L$ to the heat equation 
\begin{equation}\label{eq:freestokes}
\left\{\begin{array}{ccc}
\d_t u_L-\bar\mu\Delta u_L&=&0,\\
u_L|_{t=0}&=&u_0,
\end{array}\right.
\end{equation}
and define $u^n_L:=\dot S_n u_L.$ Of course that $\div u_0\equiv0$ implies
that $\div u_L\equiv1.$
Now,  $\bar u^{n+1}:=u^{n+1}-u^{n+1}_L$   satisfies
\begin{equation*}
\left\{\begin{array}{ccc}
\d_t\bar u^{n+1}+u^n\cdot \nabla \bar u^{n+1}-\bar{\mu}\Delta\bar u^{n+1}+\nabla Q^{n+1}&=&\bar c^n,\\
\div \bar u^{n+1}&=&0,\\
\bar u^{n+1}|_{t=0}&=&0,
\end{array}\right.\end{equation*}
with $\bar c^n=c^n-u^n\cdot\nabla u^{n+1}_L.$ 
\smallbreak
{}Note that $u^n\cdot\nabla u^{n+1}_L=\div(u^n\otimes u^{n+1}_L).$
Hence, as $\dot B^{\frac d{p_2}}_{p_2,1}$ is an algebra for $p_2<\infty,$
we have
\begin{equation}\label{sequencest:local homo,bar cn}
\|\bar c^n\|_{\dot B^{\frac d{p_2}-1}_{p_2,1}}\lesssim \|c^n\|_{\dot B^{\frac d{p_2}-1}_{p_2,1}}
+\|u^n\|_{\dot B^{\frac d{p_2}}_{p_2,1}}\|u^{n+1}_L\|_{\dot B^{\frac d{p_2}}_{p_2,1}}.
\end{equation}
Hence,  bounding $c^n$ as above but splitting 
$u^n$ into $\bar u^n+u_L^n$ 
when dealing with the terms $\nabla u^n\cdot\nabla \theta^n$ or $Du^n\cdot \nabla\theta^n,$
we get under hypothesis \eqref{eq:smallT}, for all $t\in[0,T],$ 
 $$\displaylines{\|\bar u^{n+1}\|_{\dot Y^{p_2}(t)}+\|\nabla Q^{n+1}\|_{\dot Z^{p_2}(t)}\leq C
 \Bigl(\|\theta^n\|_{\dot X^{p_1}(t)}
 \bigl(\|\theta^n\|_{\dot X^{p_1}(t)}
 + \|\bar u^{n}\|_{\dot Y^{p_2}(t)}
  + \|\nabla\bar Q^{n}\|_{\dot Z^{p_2}(t)}\bigr)\hfill\cr\hfill
  +\|u_L\|_{L^2_t(\dot B^{\frac d{p_2}}_{p_2,1})}\bigl(\|u^n\|_{L^2_t(\dot B^{\frac d{p_2}}_{p_2,1})}
  +\|\nabla\theta^n\|_{L^2_t(\dot B^{\frac d{p_1}}_{p_1,1})}\bigr)
  +\|\theta^n\|_{L^\infty_t(\dot B^{\frac d{p_1}}_{p_1,1})}\|\nabla u_L\|_{L^1_t(\dot B^{\frac d{p_2}}_{p_2,1})}\Bigr).}
$$

Therefore, if we assume in addition that $T$ has been chosen so that
  \begin{equation}\label{eq:smallT1}
  \|u_L\|_{L^2_T(\dot B^{\frac d{p_2}}_{p_2,1})\cap L^1_T(\dot B^{\frac d{p_2}+1}_{p_2,1})}
  +\|\nabla Q_L\|_{L^1_T(\dot B^{\frac d{p_2}-1}_{p_2,1})}\leq\tau
  \end{equation}
  and if 
  \begin{equation}\label{eq:smallsol}
\|\bar u^{n}\|_{\dot Y^{p_2}(T)}
  + \|\nabla Q^{n}\|_{\dot Z^{p_2}(T)}\leq\tau
\end{equation}
  then we have also (taking $\tau$ smaller if needed) \eqref{eq:smallT} and
  $$
  \|\bar u^{n+1}\|_{\dot Y^{p_2}(T)}
  + \|\nabla  Q^{n+1}\|_{\dot Z^{p_2}(T)}\leq\tau.
  $$
  Now, an elementary 
  induction argument enables us to conclude that 
  both \eqref{eq:smallT} and \eqref{eq:smallsol} are satisfied (for all $n\in\N$)
  if $T$ has been chosen so that \eqref{eq:smallT1} holds.

\subsubsection*{Step 3. Convergence of the scheme}

Let us just treat the case where only local 
existence is expected (that is $u_0$ may be large). 
We fix some time $T$ such that \eqref{eq:smallT1} is fulfilled. 
Let  $(\dt^n,\du^n,\dQ^n):=(\theta^{n+1}-\theta^n,u^{n+1}-u^n,Q^{n+1}-Q^n).$
We have 
$$
\left\{\begin{array}{ccc}
\d_t\dt^n+u^n\cdot\nabla\dt^n-\bar\kappa\Delta\dt^n&=&-\du^{n-1}\cdot\nabla\theta^n+a^{n+1}-a^n,\\[1ex]
\d_t\du^n+u^n\cdot\nabla\du^n-\bar\mu\Delta\du^n+\nabla\dQ^n&=&-\du^{n-1}\cdot\nabla u^n
+c^{n+1}-c^n,\\
(\delta\theta^n,\delta u^n)|_{t=0}&=&(\dot \Delta _n\theta_0,\dot  \Delta_n u_0).
\end{array}
\right.
$$
By arguing exactly as in the proof of the stability estimates below, it is not difficult
to establish that if $\tau$ has been chosen small enough in \eqref{eq:smallT1} then 
for all $n\geq1,$ 
$$
\displaylines{
\|\dt^n\|_{\dot X^{p_1}(T)}
+\|\du^n\|_{\dot Y^{p_2}(T)}
+\|\nabla\dQ^n\|_{\dot Z^{p_2}(T)}
\leq C\bigl(2^{n\frac{d}{p_1}}\|\dot\Delta_n\theta_0\|_{L^{p_1}}
+2^{n(\frac d{p_2}-1)}\|\dot\Delta_n u_0\|_{L^{p_2}}\bigr)
\hfill\cr\hfill+\frac12 \bigl(\|\dt^{n-1}\|_{\dot X^{p_1}(T)}
+\|\du^{n-1}\|_{\dot Y^{p_2}(T)}+\|\nabla\dQ^{n-1}\|_{\dot Z^{p_2}(T)}\bigr).}
$$

Hence $(\theta^n,u^n,\nabla Q^n)_{n\in\N}$ is a Cauchy sequence in $\dot F^{p_1,p_2}_T.$
The limit $(\theta,u,\nabla Q)$ belongs to $\dot F^{p_1,p_2}_T$ and
obviously satisfies System \eqref{eq:global}.

\subsubsection*{Step 4. Uniqueness and stability estimates}

Let us consider two solutions $(\theta^1,u^1,\nabla Q^1)$ and $(\theta^2,u^2,\nabla Q^2)$
of System \eqref{eq:global}, in the space $\dot F_T^{p_1,p_2}$
with $(p_1,p_2)$ satisfying (\ref{cond:global,p}). 
The difference 
$(\dt,\du,\dQ):=(\theta^2-\theta^1,u^2-u^1,Q^2-Q^1)$
between these two solutions satisfies
$$
\left\{\begin{array}{ccc}
\d_t\dt+u^1\cdot\nabla\dt-\bar\kappa\Delta\dt=-\du\cdot\nabla\theta^2+a(\theta^2)-a(\theta^1),\\[1ex]
\d_t\du+u^1\cdot\nabla\du-\bar\mu\Delta\du+\nabla\dQ=-\du\cdot\nabla u^2+c(\theta^2,u^2,\nabla Q^2)
-c(\theta^1,u^1,\nabla Q^1),\\[1ex]
\div\du=0.\end{array}
\right.
$$
Therefore, according to Propositions \ref{th:TDnu} and \ref{th:Stokes}, we have for all $t\in[0,T],$ 
$$\displaylines{
\|\dt\|_{\dot X^{p_1}(t)}\lesssim
e^{\|u^1\|_{L_t^1(\dot B^{\frac d{p_2}+1}_{p_2,1})}}
\Bigl(\|\dt_0\|_{\dot B^{\frac d{p_1}}_{p_1,1}}+\|\du\cdot\nabla\theta^2\|_{L_t^1(\dot B^{\frac d{p_1}}_{p_1,1})}+\|a(\theta^2)-a(\theta^1)\|_{L_t^1(\dot B^{\frac d{p_1}}_{p_1,1})}\Bigr),\cr
\|\du\|_{\dot Y^{p_2}(t)}+\|\nabla\dQ\|_{\dot Z^{p_2}(t)}
\lesssim
e^{\|u^1\|_{L_t^1(\dot B^{\frac d{p_2}+1}_{p_2,1})}}
\Bigl(\|\du_0\|_{\dot B^{\frac d{p_2}-1}_{p_1,1}}+\|\du\cdot\nabla u^2\|_{L_t^1(\dot B^{\frac d{p_2}-1}_{p_2,1})}\hfill\cr\hfill+\|c(\theta^2,u^2,\nabla Q^2)-c(\theta^1,u^1,\nabla Q^1)\|_{L_t^1(\dot B^{\frac d{p_2}-1}_{p_2,1})}\Bigr).}
$$
The nonlinear terms in the right-hand side may be handled exactly
as in the proof of the uniform estimates
(as the norms which are involved are the same, there are  no further conditions on
$p_1$ and $p_2$). 
For instance,  we have for $\frac 1{p_1}\leq \frac{1}{p_2}+\frac{1}{d},$
$$
\|\du\cdot\nabla\theta^2\|_{L_t^1(\dot B^{\frac d{p_1}}_{p_1,1})}\lesssim
\|\du\|_{L_t^2(\dot B^{\frac d{p_2}}_{p_2,1})}\|\nabla\theta^2\|_{L_t^2(\dot B^{\frac d{p_1}}_{p_1,1})}+\|\du\|_{L_t^1(\dot B^{\frac d{p_2}+1}_{p_2,1})}\|\nabla\theta^2\|_{L_t^\infty(\dot B^{\frac d{p_1}-1}_{p_1,1})},
$$
and, because
$$
\displaylines{
a(\theta^2)-a(\theta^1)=\div\Bigl(\bigl(\kappa(\vartheta^2)-\kappa(\vartheta^1)\bigr)\nabla\theta^2
+\bigl(\kappa(\vartheta^1)-\bar\kappa\bigr)\nabla\dt\Bigr)\hfill\cr\hfill
-\bigl(\kappa'(\vartheta^2)-\kappa'(\vartheta^1)\bigr)
|\nabla\theta^2|^2-\kappa'(\vartheta^1)\bigl(\nabla(\theta^1+\theta^2)\cdot\nabla\dt\bigr),}
$$
we have, according to Propositions \ref{prop:product Besov}
and \ref{prop:action Besov},
$$\displaylines{
\|a(\theta^2)-a(\theta^1)\|_{\dot B^{\frac d{p_1}}_{p_1,1}}
\lesssim C_{\theta^1,\theta^2}\Bigl(\bigl(\|\nabla\theta^1\|_{\dot B^{\frac d{p_1}}_{p_1,1}}
+\|\nabla\theta^2\|_{\dot B^{\frac d{p_1}}_{p_1,1}}\bigr)
\|\nabla\dt\|_{\dot B^{\frac d{p_1}}_{p_1,1}}
\hfill\cr\hfill
+\bigl(\|\nabla\theta^2\|_{\dot B^{\frac d{p_1}}_{p_1,1}}^2
+\|\nabla^2\theta^2\|_{\dot B^{\frac d{p_1}}_{p_1,1}}\bigr)
\|\dt\|_{\dot B^{\frac d{p_1}}_{p_1,1}}
+\|\theta^1\|_{\dot B^{\frac d{p_1}}_{p_1,1}}\|\Delta\dt\|_{\dot B^{\frac d{p_1}}_{p_1,1}}
\Bigr).}
$$
We may proceed similarly in order to bound 
the right-hand side of the inequality for $\du.$
We eventually get for all $t\in[0,T],$ 
$$
\displaylines{
\|\dt\|_{\dot X^{p_1}(t)}\leq C_{\theta^1,\theta^2,u^1}
\Bigl(\|\dt_0\|_{\dot B^{\frac d{p_1}}_{p_1,1}}+\|\theta^2\|_{\dot X^{p_1}(t)}\|\du\|_{\dot Y^{p_2}(t)}+
(\|\theta^1\|_{\dot X^{p_1}(t)}+\|\theta^2\|_{\dot X^{p_1}(t)})\|\dt\|_{\dot X^{p_1}(t)}\Bigr),\cr
\|\du\|_{\dot Y^{p_2}(t)}+\|\nabla\dQ\|_{\dot Z^{p_2}(t)}\leq C_{\theta^1,\theta^2,u^1}
\Bigl( \|\du_0\|_{\dot B^{\frac d{p_2}-1}_{p_2,1}}+
\|u^2\|_{L^2_t(\dot B^{\frac d{p_2}}_{p_2,1})}(\|\dt\|_{\dot X^{p_1}(t)}+\|\du\|_{\dot Y^{p_2}(t)})
\hfill\cr\hfill
+\bigl(\|\nabla u^2\|_{L^1_t(\dot B^{\frac{d}{p_2}}_{p_2,1})}+\|\nabla Q^2\|_{L^1_t(\dot B^{\frac{d}{p_2}-1}_{p_2,1})}\bigr)\|\dt\|_{\dot X^{p_1}(t)}
\hfill\cr\hfill
+\bigl(\|\theta^1\|_{\dot X^{p_1}(t)}+\|\theta^2\|_{\dot X^{p_1}(t)}\bigr)
\bigl(\|\dt\|_{\dot X^{p_1}(t)}+\|\du\|_{\dot Y^{p_2}(t)}+ \|\nabla\dQ\|_{\dot Z^{p_2}(t)}\bigr)\Bigr).}
$$

In the case where  $(\theta^1,u^1,\nabla Q^1)$ and $(\theta^2,u^2,\nabla Q^2)$ are  small enough on $[0,T],$ 
all the terms involving $(\dt,\du)$ in the right-hand side may be absorbed by 
the left-hand side.
This yields stability estimates on the whole interval $[0,T],$
 and implies uniqueness. 

The case where the velocity is  large requires more care
for it is not clear that the terms corresponding to $\div(\dt\cdot\nabla u^2)$, $\nabla u^2\cdot\nabla\dt,$ $\du\cdot\nabla u^2$
and $\dt\nabla Q^2$ 
are small compared to  
$\|\dt\|_{\dot X^{p_1}(t)}+\|\du\|_{\dot Y^{p_2}(t)}+\|\nabla\dQ\|_{\dot Z^{p_2}(t)}.$
However, we notice that 
they may be bounded in $L^1_t(\dot B^{\frac d{p_2}-1}_{p_2,1})$ by
$$
(\|\nabla\dt\|_{L_t^2(\dot B^{\frac d{p_1}}_{p_1,1})}
+\|\du\|_{L_t^2(\dot B^{\frac d{p_2}}_{p_2,1})})\|u^2\|_{L_t^2(\dot B^{\frac d{p_2}}_{p_2,1})}+\|\dt\|_{L_t^\infty(\dot B^{\frac d{p_1}}_{p_1,1})}
\bigl(\|\nabla u^2\|_{L_t^1(\dot B^{\frac d{p_2}}_{p_2,1})}+\|\nabla Q^2\|_{L_t^1(\dot B^{\frac d{p_2}-1}_{p_2,1})}\bigr).
$$
Obviously the terms corresponding to $u^2$ and $\nabla Q^2$ go to zero 
when $t$ tends to $0.$ If both solutions coincide 
initially, this implies uniqueness on a small enough time interval.
Then uniqueness on the whole interval $[0,T]$ follows from standard continuity arguments. 

For proving stability estimates, one may further decompose $u^1$ 
and $u^2$ into
$$
u^1=\bar u^1+u_L\quad\hbox{and}\quad 
u^2=\bar u^2+u_L, 
$$
where $u_L$ stands for the free solution to the Stokes system
that has been defined in \eqref{eq:freestokes}.
We can thus write
$$\begin{array}{lll}
\| u^2\|_{L_t^2(\dot B^{\frac d{p_2}}_{p_2,1})\cap L_t^1(\dot B^{\frac d{p_2}+1}_{p_2,1})}&\leq&\|u_L\|_{L_t^2(\dot B^{\frac d{p_2}}_{p_2,1})
\cap L_t^1(\dot B^{\frac d{p_2}+1}_{p_2,1})}+\|\bar u^2\|_{L_t^2(\dot B^{\frac d{p_2}}_{p_2,1})\cap L_t^1(\dot B^{\frac d{p_2}+1}_{p_2,1})}.
\end{array}
$$
If $T$ has been chosen so that \eqref{eq:smallT1} holds true
and if $(\theta^2,u^2,\nabla Q^2)$ is the solution that 
has been constructed above then 
we conclude that the above terms may be bounded by $\tau.$
So they may be absorbed by the left-hand side, and it is thus possible
to get the continuity of the flow map on $[0,T]$ for $T$ satisfying \eqref{eq:smallT1}. 
The details are left to the reader.


\section{The proof of Theorem \ref{th:local}}\label{s:large} \setcounter{equation}{0}

In this section we establish local  well-posedness results in the fully nonhomogeneous case: 
we just assume that  the initial temperature is positive and tends
 to some positive  constant at infinity (we  take $1$ for notational simplicity). 
In this framework, the estimates for the linear equations considered in Section \ref{s:small}
are not sufficient to bound the solutions to \eqref{equation:local} 
even at small time. The reason why is that 
 some   quadratic terms such as $\nabla u\cdot\nabla \theta$ or
$\theta\nabla Q$ may be of the same order as the terms 
of the left-hand side hence cannot be absorbed any longer.
In the fully nonhomogeneous case, the appropriate 
linear equations that have to be considered have \emph{variable coefficients}
in their main order terms. 

The first part of this section is devoted to the presentation 
and the proof of new  a priori estimates for these linear equations. 
As we believe this type of estimates to be of interest
in other contexts, we provide the statements for a wider
range of Lebesgue and regularity exponents than
 those which will be needed to establish the well-posedness of \eqref{equation:local}
 in our functional framework.   The second part of this section is devoted to
  the proof of Theorem \ref{th:local}.

\subsection{The linearized equations}

In order to bound the temperature, we shall establish a priori estimates
in nonhomogeneous Besov norms for the solutions to 
 \begin{equation}\label{lineareq:energy}
\left\{\begin{array}{ccc}\d_{t}\theta+q\cdot\nabla \theta -\div(\kappa\nabla \theta)&=&f,\\
\theta|_{t=0}&=&\theta_0.
\end{array}\right.
\end{equation}
 
 Our main result (which extends the corresponding one in \cite{Danchin01local}) reads: 
 \begin{prop}\label{p:energy}
 Let $\theta$ satisfy \eqref{lineareq:energy} on $[0,T]\times\R^d.$
 Let $(p_1,p_2)\in(1,\infty)^2$ and $s\in\R$ fulfill
 \begin{equation}\label{eq:p1,p2}
  -1-d\min\biggl\{\frac{1}{p_1},\frac{1}{p'_1},\frac{1}{p_2}\biggr\}<s
  \leq d\min\biggl\{\frac{1}{p_1},\frac{1}{p_2}+\frac1d\biggr\}\cdotp
  \end{equation}
   
Suppose that the conductivity  function $\kappa,$ the divergence free  vector-field
$q$, the initial data $\theta_0$ and the source term $f$ are smooth enough and decay  at infinity,  and that
\begin{equation}\label{eq:loinduvide}
m:=\min_{(t,x)\in[0,T]\times\R^d} \kappa(t,x)>0.
\end{equation}
Then there exist constants $c_{1,p_1}(d,p_1,m),C_{1,p_1}(d,p_1,p_2,s,m),\tilde C_{1,p_1}(d,p_1,s,m)$ such that the solution to \eqref{lineareq:energy} satisfies for all $t\in[0,T]:$
\begin{equation}\label{linearest:energy,p}
\begin{split}
\|\theta\|_{\tilde L^\infty_t(B^{s}_{p_1,1})}+c_{1,p_1}\|\theta\|_{L^1_t(B^{s+2}_{p_1,1})}
\leq& e^{C_{1,p_1}(\|\nabla q\|_{L^{1}_t(B^{d/{p_2}}_{p_2,1})}  +\|\nabla \kappa\|_{L^2_t(B^{d/{p_1}}_{p_1,1})}^2)}\\
&\times(\|\theta_0\|_{B^{s}_{p_1,1}}+\tilde C_{1,p_1}\|\Delta_{-1}\theta\|_{L^1_t(L^{p_1})}
+\|f\|_{L^1_t(B^{s}_{p_1,1})}).
\end{split}
\end{equation}
\end{prop}

\begin{proof}
As a warm up, we focus on the special case $p_1=p_2=2$ and $s\in(-d/2,d/2]$
which may be achieved by classical energy arguments. 
Applying $\dj$ to (\ref{lineareq:energy}) yields for all $j\geq -1$,
\begin{equation}\label{eq:thetaj}\d_t\theta_j+q\cdot\nabla\theta_j
-\div(\kappa\nabla\theta_j)
=f_j+R^1_j-\div R^2_j,\end{equation}
where
$$
\theta_j=\dj \theta,
\quad f_j=\dj f,
\quad R^1_j=[q,\dj]\cdot\nabla\theta\ \hbox{ and }\
R^2_j=[\kappa,\dj]\nabla \theta.
$$

Taking the $L^2$ inner product of the above equation with $\theta_j$ and integrating by parts (recall that  $\div q=0$), we get
$$\frac{1}{2}\frac{d}{dt}\|\theta_j\|^2_{L^2}
+\Int \kappa|\nabla\theta_j|^2
\leq \|\theta_j\|_{L^2}(\|f_j\|_{L^2}+\|R^1_j\|_{L^2}+\|\div R^2_j\|_{L^2}).$$

Notice that we have  $\|\nabla\theta_j\|_{L^2}\approx 2^j\|\theta_j\|_{L^2}$ for $j\geq 0$. Hence, dividing formally both sides of the inequality by $\|\theta_j\|_{L^2}$ and integrating with respect to the time variable, we get for  some constant $c_1$ depending only on $d,$
\begin{align*}
\|\theta_j\|_{L^\infty_t(L^2)}+c_1m2^{2j}\|\theta_j\|_{L^1_t(L^2)}\leq &
\|(\theta_0)_j\|_{L^2}+\delta^{-1}_{j}c_1m2^{2j}\|\Delta_{-1}\theta\|_{L^1_t(L^2)}\\
&+\|f_j\|_{L^1_t(L^2)}+\|R^1_j\|_{L^1_t(L^2)}+\|\div R^2_j\|_{L^1_t(L^2)},
\end{align*}

where
$$
\delta^{-1}_j=1\hbox{ if }j=-1\quad\hbox{and}\quad 
\delta^{-1}_j=0\hbox{ if }j\not=-1.
$$

Applying  Inequality \eqref{est:com3} with regularity index $s$
and Inequality \eqref{est:com4} with regularity index $s+1$ and $\nu=1$
yields
\begin{align*}
\|R^1_j\|_{L^1_t(L^2)}\lesssim& \,2^{-js}c_j\Int^t_0 \|\nabla q\|_{ B^{d/2}_{2,1}}\|\nabla\theta\|_{ B^{s-1}_{2,1}}\, dt'\quad\hbox{if }\ -d/2<s\leq d/2+1,\\
\|\div R^2_j\|_{L^1_t(L^2)} \lesssim& \,2^{-js}c_j\Int^t_0\|\nabla \kappa\|_{B^{d/2}_{2,1}}\| \nabla\theta\|_{B^{s}_{2,1}}\, dt'
\quad\hbox{if }\ -d/2-1<s\leq d/2.
\end{align*}

Now multiplying both sides by $2^{js}$, summing up over $j$ and taking advantage of the interpolation inequality $\|\cdot\|_{B^{s+1}_{2,1}}\lesssim \|\cdot\|_{B^s_{2,1}}^{1/2}\|\cdot\|_{B^{s+2}_{2,1}}^{1/2}$ in Proposition \ref{p:Besov} yields
\begin{align*}
\|\theta\|_{\tilde L^\infty_t({B}^{s}_{2,1})}+c_1m\|\theta\|_{  L^1_t({B}^{s+2}_{2,1})}\leq&
\|\theta_0\|_{ B^{s}_{2,1}}+c_1m\|\Delta_{-1}\theta\|_{L^1_t(L^2)}
+\|f\|_{  L^1_t( {B}^{s}_{2,1})}\\
&+C_1\Int^t_0(\|\nabla q\|_{ {B}^{d/2}_{2,1}} 
 +\|\nabla \kappa\|_{B^{d/2}_{2,1}}^2)\|\theta\|_{  {B}^{s}_{2,1}}\,dt'.
\end{align*}
Then applying Gronwall's inequality leads to Inequality \eqref{linearest:energy,p}.
\medbreak
To treat the general  case  $1<p_1<\infty$  we multiply
 \eqref{eq:thetaj} by 
 $|\theta_j|^{p_1-2}\theta_j.$ We arrive at
\begin{align*}
\frac{1}{p_1}\frac{d}{dt}\Int |\theta_j|^{p_1}\,dx+(p_1-1)\Int \kappa |\theta_j|^{p_1-2}|\nabla \theta_j|^2\,dx
\leq &\|\theta_j\|_{L^{p_1}}^{p_1-1}(\|f_j\|_{L^{p_1}}+\|R^1_j\|_{L^{p_1}}+\|\div R^2_j\|_{L^{p_1}}).
\end{align*}

Next, we use (bearing in mind that $1<p_1<\infty$) the following Bernstein type inequality (see Appendix B in \cite{Danchin10euler}):
\begin{equation}\label{equality:bernstein}
\Int |\theta_j|^{p_1-2}|\nabla\theta_j|^2
\gtrsim 2^{2j}\int |\theta_j|^{p_1}\,dx\quad\hbox{for }\  j\geq 0.
\end{equation}

Hence we get
\begin{align*}
\frac{d}{dt}\|\theta_j\|_{L^{p_1}}^{p_1}+2^{2j}m\|\theta_j\|_{L^{p_1}}^{p_1}
&\lesssim  \|\theta_j\|_{L^{p_1}}^{p_1-1}(\|f_j\|_{L^{p_1}}+\|R^1_j\|_{L^{p_1}}
+\|\div R^2_j\|_{L^{p_1}}).
\end{align*}
Therefore dividing both sides by $\|\theta_j\|_{L^{p_1}}^{p_1-1}$
and using that, according to \eqref{est:com3} and \eqref{est:com4},
\begin{align*}
\|R^1_j\|_{L^{p_1}}\lesssim& \,2^{-js}c_j\|\nabla q\|_{ B^{d/{p_2}}_{p_2,1}}\|\nabla\theta\|_{ B^{s-1}_{p_1,1}}\quad\hbox{if }\ -d\min\biggl(\frac1{p'_1},\frac1{p_2}\biggr)<s\leq1+d\min\biggl(\frac1{p_1},\frac1{p_2}\biggr),\\
\|\div R^2_j\|_{L^{p_1}} \lesssim& \,2^{-js}c_j\|\nabla \kappa\|_{B^{d/{p_1}}_{p_1,1}}
\|\nabla\theta\|_{B^{s}_{p_1,1}}\quad\hbox{if }\ -d\min\biggl(\frac1{p'_1}\frac1{p_1}\biggr)
<s+1\leq1+\frac d{p_1},
\end{align*}
 integrating in  time, multiplying both sides with $2^{js}$, summing up over $j\in \Z$  and performing an interpolation inequality, one arrives at
\begin{align*}
\|\theta\|_{\tilde L^\infty_t(B^{s}_{p_1,1})}+c_{1,p_1}m\|\theta\|_{L^1_t(B^{s+2}_{p_1,1})}
\leq& \|\theta_0\|_{B^{s}_{p_1,1}}+c_{1,p_1}2^{-(s+2)}m\|\Delta_{-1}\theta\|_{L^1_t(L^{p_1})}
+\|f\|_{L^1_t(B^{s}_{p_1,1})}\\
&+C_{1,p_1}\Int^t_0 (\|\nabla q\|_{B^{d/{p_2}}_{p_2,1}} +\|\nabla \kappa\|_{B^{d/{p_1}}_{p_1,1}}^2)\| \theta\|_{B^{s}_{p_1,1}}\, dt',
\end{align*}
which yields (\ref{linearest:energy,p}) by Gronwall inequality, except
in the case where $s$ is too negative.
\smallbreak
To improve the condition over $s$ for $s$ negative, it suffices to use the 
fact that, owing to $\div q=0,$ one has
$$
R_j^1=\div([q,\dj]\theta).
$$
Then one may  apply Inequality \eqref{est:com4} to $\div[q,\dj]\theta$
with $s+1$ instead of $s.$ The details are left to the reader.
\end{proof}

 \begin{rem}
 Let us further remark that 
 $$
\div  R_j^2=[\nabla \kappa,\Delta_j]\cdot\nabla\theta+[\kappa,\Delta_j]\Delta\theta.
$$
This decomposition allows to improve the condition \eqref{eq:p1,p2} for 
positive $s$:  if  we only assume that $s\leq 1+d\min(1/p_1,1/p_2)$
then Inequality \eqref{linearest:energy,p} holds true
 with the additional term 
 $\|\nabla^2 \kappa\|_{L^1_t(B^{d/{p_1}}_{p_1,1})}$ in the exponential. 
 As only the case $s=d/p_1$  is needed for proving Theorem \ref{th:local},
we do not provide more details here. 

Note also that, for 
 $s=d/{p_1}$, Condition  \eqref{eq:p1,p2} holds if and only if 
$1/{p_1}\leq 1/{p_2}+1/d$.
\end{rem}


In the fully nonhomogeneous case, the appropriate linearized 
momentum equation turns out to be
  \begin{equation}\label{lineareq:momentum}\left\{\begin{array}{ccc}
\d_{t}u+q\cdot\nabla u-\div(\mu\nabla u)+P\nabla Q&=&h,\\
\div u&=&0,\\
u|_{t=0}&=&u_0,
\end{array}\right.
\end{equation}
where $q$ is a given divergence free vector-field, and $(\mu,P)$ are
given positive functions. 
Let us 
first consider the case $p_1=p_2=2$ which may be handled by standard
energy arguments.

\begin{prop}\label{p:momentum}
Let $P,$ $\mu,$ $q$, $h,$ $u_0$ be smooth and decay sufficiently at infinity with $\div q=0.$
Let  $(u,\nabla Q)$ satisfy \eqref{lineareq:momentum} on $[0,T]\times\R^d.$  
Let $s\in[0,d/2].$ Suppose that, for some positive constants $m$ and $M,$
we have
\begin{equation}\label{condition:P}
\|\nabla P\|_{L^\infty_T(B^{d/2-1}_{2,1})}+\|P\|_{L^\infty([0,T]\times\R^d)}\leq M
\quad\hbox{and}\quad  \min(P,\mu)\geq m.
\end{equation}
There exists a constant $c_P=c_P(d,s)$ such that, 
if for some integer $N,$ one has
 \begin{equation}\label{condition:N}
\inf_{x\in \R^d,t\in [0,T]}S_NP(t,x)\geq m/2,\quad \|P-S_NP\|_{L^\infty_T(B^{d/2}_{2,1})}\leq c_Pm,
 \end{equation}
then there exist constants $c_2(d,m)$, $C_2(d,s,m,M,N)$, $\tilde C_2(d,s,m)$, $C_3(d,s,m,M,N)$,  such that for all $t\in[0,T],$
\begin{equation}\label{linearest:momentum}\begin{split}
\|u\|_{\tilde L^\infty_t( {B}^{s}_{2,1})}+c_2\|u\|_{ L^1_t( {B}^{s+2}_{2,1})}
\leq& e^{C_2(\|\nabla q\|_{L^1_t( {B}^{d/2}_{2,1})}+\|\nabla\mu\|_{L^2_t( B^{d/2}_{2,1})}^2 )}\\
&\times\bigl(\|u_0\|_{  B^{s}_{2,1}}+\tilde C_2\|\Delta_{-1}u\|_{L^1_t(L^2)}+C_2\|h\|_{L^1_t( {B}^{s}_{2,1})}\bigr),
\end{split}
\end{equation}

\begin{equation}\label{linearest:Q}
 \|\nabla Q\|_{L^1_t( {B}^{s}_{2,1})}\leq C_3
 \int_0^t\biggl(\|h\|_{B^s_{2,1}}+\|\nabla q\|_{{B}^{d/2}_{2,1}}\|u\|_{B^s_{2,1}}
+\|\nabla\mu\|_{B^{d/2}_{2,1}}\|\nabla u\|_{{B}^{s}_{2,1}}\biggr)\,d\tau.
\end{equation}
\end{prop}

\begin{proof}
Following the proof of Proposition \ref{p:energy}, we apply $\dj$ to 
(\ref{lineareq:momentum}). This  yields
$$\d_t u_j+q\cdot\nabla u_j-\div(\mu\nabla u_j)
=h_j+R^1_j-\div R^2_j-\dj( P\nabla Q),$$
where
$$u_j:=\dj u,\quad h_j:=\dj h,\quad R^1_j:=[q,\dj]\cdot\nabla u,\quad  
R^2_j:=[\mu,\dj]\cdot\nabla u.$$

As above, we thus get if $-d/2<s\leq d/2,$
\begin{eqnarray}
\|u\|_{\tilde L^\infty_t( {B}^{s}_{2,1})}+c_2(d)m\|u\|_{L^1_t( {B}^{s+2}_{2,1})}\leq
\|u_0\|_{ B^{s}_{2,1}}+\tilde C_2(c_2,s)\|\Delta_{-1}u\|_{L^1_t(L^2)}
+\|h\|_{L^1_t( {B}^{s}_{2,1})}\nonumber\\\label{eq:u}
+C(d,s,m)\Int^t_0\bigl(\|\nabla q\|_{ {B}^{d/2}_{2,1}}\|u\|_{  {B}^{s}_{2,1}}
+\|\nabla\mu\|_{ {B}^{d/2}_{2,1}}^2
\|\nabla u\|_{  {B}^{s}_{2,1}}\bigr)\,dt'+\|P\nabla Q\|_{L^1_t( {B}^{s}_{2,1})}.
\end{eqnarray}

We now have to bound $\nabla Q$. Applying the divergence operator to the first equation, we then arrive at the following  elliptic equation 
\emph{with variable coefficients}\footnote{Here we use that $\div(q\cdot\nabla u)=\div(u\cdot\nabla q)$ and $\div\bigl(\div(\mu\nabla u)\bigr)=\div\bigl(\nabla u\cdot\nabla\mu\bigr)$ owing
to $\div u=\div q=0$.}: 
\begin{equation}\label{lineareq:Q}
\div (P\nabla Q)=\div L \quad\textrm{ with }\quad L:=-u\cdot \nabla q
+\nabla\mu\cdot (\nabla u)^{T}+h.
\end{equation}
First we take the $L^2$ inner product to the equation (\ref{lineareq:Q}) with $Q$ to get
\begin{equation}\label{eq:LM}
m\|\nabla Q\|_{L^2}\leq \|L\|_{L^2}.
\end{equation}

Next, applying $\dj$ to (\ref{lineareq:Q}) yields (with obvious notation)
$$
\div(P\nabla Q_j)=\div L_j+\div ([P,\dj]\nabla Q).
$$
Hence, taking the $L^2$ inner product with $Q_j$  and  integrating
by parts yields
$$
m\|\nabla Q_j\|_{L^2}\leq \|L_j\|_{L^2}+\|[P,\dj]\nabla Q\|_{L^2}. 
$$
So using the commutator estimate \eqref{est:com1}, 
we easily get if  $-d/2<\nu\leq 1$ and $-d/2<s\leq \nu+d/2,$
\begin{equation}\label{linearestimate:Q,general}
m\|\nabla Q\|_{L^1_t(B^s_{2,1})}\leq \|L\|_{L^1_t(B^s_{2,1})}
+C_Q(d,s,\nu)\|\nabla P\|_{L^\infty_t(B^{d/2+\nu-1}_{2,1})}\|\nabla Q\|_{L^1_t(B^{s-\nu}_{2,1})}.
\end{equation}

Now we consider two cases:
\begin{itemize}
\item Case $0<s\leq d/2.$
Let us first assume that  $\nabla P$ has some extra regularity: 
suppose for instance that it belongs to
 ${ L^\infty_T(B^{d/2+\nu-1}_{2,1})}$ for some $\nu$ 
 such that $\nu+d/2\geq s>\nu>0.$
As  $\|\cdot\|_{B^0_{2,2}}=\|\cdot\|_{L^2}$  we arrive (by interpolation) at
\begin{equation}\label{eq:interpo1}
\|\nabla Q\|_{B^{s-\nu}_{2,1}}\lesssim\|\nabla Q\|_{L^2}^{\nu/s}
\|\nabla Q\|_{B^{s}_{2,1}}^{1-\nu/s}\lesssim \|L\|_{L^2}^{\nu/s}\|\nabla Q\|_{B^{s}_{2,1}}^{1-\nu/s}.
\end{equation}
Hence (\ref{linearestimate:Q,general}) implies that
\begin{equation}\label{linearestimate:Q,s>0}
\|\nabla Q\|_{L^1_t(B^s_{2,1})}\leq C(d,s,\nu,m)
(1+\|\nabla P\|_{L^\infty_t(B^{d/2+\nu-1}_{2,1})})^{s/\nu}\|L\|_{L^1_t(B^s_{2,1})}.
\end{equation}

Now, if $P$ satisfies only Conditions \eqref{condition:P}
and \eqref{condition:N} then  we decompose it
 into
$$
P=P_N+(P-P_N) \quad\hbox{with } P_N:=S_NP.
$$
Note that  $\nabla P_N\in H^\infty$ and that the equation for $Q$ recasts in 
$$
\div(P_N\nabla Q)=\div (L+ E_N),\quad \textrm{where}\quad E_N=(P_N-P)\cdot \nabla Q.
$$

Therefore, following the  procedure  leading to \eqref{linearestimate:Q,general}
and bearing  the first part of Condition \eqref{condition:N} in mind, yields 
\begin{equation}\label{eq:Q1}
\frac m2\|\nabla Q\|_{ L^1_t(B^s_{2,1})}\leq  \|L+E_N\|_{L^1_t(B^s_{2,1})}
+C_Q(d,s,\nu)\|\nabla P_N\|_{L^\infty_t(B^{d/2+\nu-1}_{2,1})}\|\nabla Q\|_{L^1_t(B^{s-\nu}_{2,1})}.
\end{equation}

We notice that for $-d/2<s\leq d/2$,
$$
\begin{array}{lll}
\|\nabla P_N\|_{ L^\infty_t(B^{d/2+\nu-1}_{2,1})}&\leq&
C_P(d)2^{N\nu}\|\nabla P\|_{  L^\infty_t(B^{d/2-1}_{2,1})},\\[1ex]
\|E_N\|_{  L^1_t(B^{s}_{2,1})}&\leq& C_P(d,s)\|P_N-P\|_{ L^\infty_t(B^{d/2}_{2,1})}
\|\nabla Q\|_{ L^1_t(B^{s}_{2,1})},\end{array}
$$
hence the term pertaining to $E_N$ may be absorbed by the left-hand side
of \eqref{eq:Q1} if $c_P$ is small enough in  \eqref{condition:N}. Then using the same
interpolation argument as above,  we
end up with 
\begin{equation}\label{eq:Q2}
 m\|\nabla Q\|_{ L^1_t(B^s_{2,1})}\leq  C_Q(d,s,\nu) 2^{N\nu}
 \bigl(1+\|\nabla P\|_{L^\infty_t(B^{d/2-1}_{2,1})}\bigr)^{s/\nu}
  \|L\|_{L^1_t(B^{s}_{2,1})}.
\end{equation}

Now, in order to complete the proof of Inequality (\ref{linearest:Q}), it
is only a matter of using the  product estimates (ii) stated
in Proposition \ref{prop:product Besov} for bounding $L,$ which implies that
$$
\|L\|_{B^s_{2,1}}\lesssim\|\nabla q\|_{B^{\frac d2}_{2,1}}\|u\|_{B^s_{2,1}}
+\|\nabla\mu\|_{B^{\frac d2}_{2,1}}\|\nabla u\|_{B^s_{2,1}}+\|h\|_{B^s_{2,1}}
\quad\hbox{if }\ -d/2<s\leq d/2.
$$

\item Case $s=0$: in this case,  the interpolation inequality \eqref{eq:interpo1} fails,
so that we have to modify the proof accordingly. 
First we apply  Inequality (\ref{linearestimate:Q,general}) for some  $0<\nu< 1$, 
and Inequality \eqref{est:com1}, to get (by virtue of \eqref{eq:LM}):
\begin{align*}
\|[P,\dj]\nabla Q\|_{L^1_t(B^0_{2,1})}\lesssim& \,c_j\|\nabla P\|_{L^\infty_t(B^{d/2+\nu-1}_{2,1})}
\|\nabla Q\|_{L^1_t(B^{0-\nu}_{2,1})}\\
\lesssim&\, c_j\|\nabla P\|_{L^\infty_t(B^{d/2+\nu-1}_{2,1})}\|L\|_{L^1_t(L^2)},
\end{align*}
 hence
\begin{equation}\label{linearestimate:Q,s=0}\begin{split}
\|\nabla Q\|_{L^1_t(B^0_{2,1})}&\lesssim \|L\|_{L^1_t(B^0_{2,1})}
+\|\nabla   P\|_{L^\infty_t(B^{d/2+\nu-1}_{2,1})}\|L\|_{ L^1_t(L^2)}\\
&\leq C(d,\nu,m)(1+\|\nabla   P\|_{L^\infty_t(B^{d/2+\nu-1}_{2,1})})\|L\|_{L^1_t(B^0_{2,1})},
\end{split}\end{equation}

which is quite similar  as (\ref{linearestimate:Q,s>0})  and hence the same procedure implies also \eqref{eq:Q2}.
\end{itemize}

In order to prove \eqref{linearest:momentum}, it suffices
to plug the above estimate for the pressure in \eqref{eq:u}.  
The main point is that, if $-d/2<s\leq d/2$ then we have
$$
\|P\nabla Q\|_{B^s_{2,1}}\lesssim\bigl(\|P\|_{L^\infty}+\|\nabla P\|_{B^{\frac d2-1}_{2,1}}\bigr)
\|\nabla Q\|_{B^s_{2,1}},
$$
as may be easily seen by  decomposing  $P\nabla Q$ into $\Delta_{-1}P\:\nabla Q+({\rm Id}-\Delta_{-1})P\:\nabla Q$ and using
 the product estimates of Proposition \ref{prop:product Besov}.
Then  Gronwall lemma leads to Inequality (\ref{linearest:momentum}).
\end{proof}

\begin{rem}
In Proposition \ref{p:momentum} we need the assumption $s\geq 0$ to get
 the  necessary $L^2$ estimate for $\nabla Q$. However, some negative indices may be achieved by duality arguments. As the corresponding estimates are not needed in our paper, 
 we here do not give more details on that issue. 
\end{rem}


We now want to extend Proposition  \ref{p:momentum} to more general Besov spaces
which are not directly related to the energy space. 
To simplify the presentation, we focus on the regularity exponent $s=d/p_2-1$ which is the only one that
we will have to consider in the proof of Theorem \ref{th:local}.
Our main result reads:

\begin{prop}\label{p:momentum,p}
Let $T>0$ and $(u,\nabla Q)$ be a solution to \eqref{lineareq:momentum} on $[0,T]\times\R^d.$  
Suppose that the given functions  $P,$ $\mu,$  that 
the divergence free vector-field $q$, the initial data $u_0$ and the source term $h$ are smooth and decay at infinity. 
 Let $p_1$ be in $[1,\infty)$ and $p_2\in[2,4]$ satisfy
\begin{equation}\label{condition:p,r}
 p_2\leq\frac{2p_1}{p_1-2}\ \hbox{ if }\ p_1>2,\quad
  \frac{1}{p_2}\leq \frac{1}{p_1}+\frac{1}{d},
\quad\hbox{and}\quad (p_1,p_2)\not=(4,4)\ \hbox{ if }\ d=2.
 \end{equation}

Assume that there exist some constants $0<m<M,$ $c_{P,p_1,p_2}(d,p_1,p_2)$ small enough,
and $N\in\N$ such that
\begin{equation}\label{condition:N,p}
\begin{array}{c}
\min(\mu,P)\geq m,\quad \|P\|_{L^\infty([0,T]\times\R^d)}
+\|\nabla P\|_{L^\infty_T(B^{d/{p_1}-1}_{p_1,1})}\leq M,\\[1ex]
\inf_{x\in \R^d,t\in [0,T]}S_NP(t,x)\geq m/2,\quad \|P-S_NP\|_{L^\infty_T(B^{d/{p_1}}_{p_1,1})}\leq c_{P,p_1,p_2}m.\end{array}
 \end{equation}

Then there exist constants $c_{2,p_2}(d,p_2,m),C_{2,p_1,p_2}(d,p_1,p_2,m,M,N),\tilde C_{2,p_2}(d,p_2,m)$, \\ $C_{3,p_1,p_2}(d,p_1,p_2,m,M,N)$ such that the following a priori estimates hold:

\begin{equation}\label{linearest:momentum,p}
\begin{split}
\|u\|_{\tilde L^\infty_t(B^{d/{p_2}-1}_{p_2,1})}&+c_{2,p_2}\|u\|_{L^1_t(B^{d/{p_2}+1}_{p_2,1})}
\\
& \leq e^{C_{2,p_1,p_2}(\|\nabla q\|_{L^1_t(B^{d/{p_2}}_{p_2,1})}+\|\nabla q\|_{L^{4/3}_t(B^{d/{p_2}-1/2}_{p_2,1})}^{4/3}
+\|\nabla\mu\|_{L^2_t(B^{d/{p_1}}_{p_1,1})}^2
+\|\nabla\mu\|_{L^{2/(1-\eta)}_t(B^{d/{p_1}-\eta}_{p_1,1})}^{2/(1-\eta)})}\\
&\times(\|u_0\|_{B^{d/{p_2}}_{p_2,1}}+\tilde C_{2,p}\|\Delta_{-1}u\|_{L^1_t(L^{p_2})}
+C_{2,p_1,p_2}\|h\|_{L^1_t(L^2\cap B^{d/{p_2}-1}_{p_2,1})}),
\end{split}
\end{equation}

with $\eta=\min(1/2,d/p_1),$ and 
\begin{equation}\label{linearest:Q,p}
\begin{split}
\|\nabla Q\|_{L^1_t(B^{\frac{d}{p_2}-1}_{p_2,1}\cap L^2)}
\leq C_{3,p_1,p_2}\int_0^t&\biggl(\|h\|_{L^2\cap B^{d/{p_2}-1}_{p_2,1}}
+\|\nabla q\|_{B^{d/{p_2}}_{p_2,1}}\|u\|_{B^{d/{p_2}-1}_{p_2,1}}
+\|u\cdot\nabla q\|_{L^2}\\&+\|\nabla\mu\|_{B^{d/{p_1}}_{p_1,1}}
\|\nabla u\|_{B^{d/{p_2}-1}_{p_2,1}}
+\|\nabla u\cdot\nabla\mu\|_{L^2}\biggr)d\tau.
\end{split}
\end{equation}
\end{prop}
\begin{proof}

With the notation of  Proposition \ref{p:energy}, we have 
\begin{eqnarray}
\label{linearest:R1,p}
&&\|(2^{j(d/{p_2}-1)}\|R^1_j\|_{L^{p_2}})_{j\in \Z}\|_{\ell^1}\lesssim \|\nabla q\|_{B^{d/{p_2}}_{p_2,1}}
\|u\|_{B^{d/{p_2}-1}_{p_2,1}},\\
\label{linearest:R2,p}
&&\|(2^{j(d/{p_2}-1)}\|\div R^2_j\|_{L^{p_2}})_{j\in \Z}\|_{\ell^1}\lesssim
\|\nabla \mu\|_{B^{d/{p_1}}_{p_1,1}}\|\nabla u\|_{B^{d/{p_2}-1}_{p_2,1}},\\
\label{linearest:P,Q,p}
&&\|P\nabla Q\|_{B^{d/{p_2}-1}_{p_2,1}}\lesssim
\bigl(\|P\|_{L^\infty}+\|\nabla P\|_{B^{d/{p_1}-1}_{p_1,1}}\bigr)\|\nabla Q\|_{B^{d/{p_2}-1}_{p_2,1}}.
\end{eqnarray}
Indeed\footnote{
Here   it is understood that   the quote marks designate the indices in the original inequalities
 \eqref{est:product Besov1}, \eqref{est:com1} and \eqref{est:com4}.}
 Inequality \eqref{linearest:R1,p} follows from  \eqref{est:com1}
 with $``p_1"=``p_2"=p_2,\quad ``\nu"=1,\quad ``s"=d/{p_2}-1$
 (note that the condition   $p_2<2d$ is not required for $\div q=0,$ a consequence
 of \eqref{est:com4} with $``s"=d/p_2$ and $``\nu"=1$)
  while \eqref{linearest:R2,p} stems from \eqref{est:com4}
 with $``p_1"=p_2,\quad ``p_2"=p_1,\quad ``\nu"=1,\quad ``s"=d/{p_2}$
 (here we need that $1/{p_2}\leq 1/{p_1}+1/d$);
 and \eqref{linearest:P,Q,p} is a consequence of 
 the decomposition $P=\Delta_{-1}P+({\rm Id}-\Delta_{-1})P$ and of 
 \eqref{est:product Besov1}  
 with 
 $``s_1"=d/{p_1},\quad ``s_2"=d/{p_2}-1$
 (which requires that $1/{p_2}\leq 1/{p_1}+1/d$ and $1/{p_1}+1/{p_2}>1/d$). 
\medbreak
Now, granted with the above inequalities, the same procedure as in Proposition \ref{p:energy} yields
\begin{equation}\label{linearest:u,p,Q}\begin{split}
\|u\|_{\tilde L^\infty_t(B^{d/{p_2}-1}_{p_2,1})}+&c_{2,p_2}\|u\|_{L^1_t(B^{d/{p_2}+1}_{p_2,1})}
\leq \|u_0\|_{B^{d/{p_2}-1}_{p_2,1}}+c_{2,p_2}2^{-(d/{p_2}+1)}\|\Delta_{-1}u\|_{L^1_t(L^{p_2})}\\
&+C(d,p_1,p_2,m)\Int ^t_0 (\|\nabla q\|_{B^{d/{p_2}}_{p_2,1}}+\|\nabla\mu\|_{B^{d/{p_1}}_{p_1,1}}^2)\|u\|_{B^{d/{p_2}-1}_{p_2,1}}\,dt'\\
&+\|h\|_{L^1_t(B^{d/{p_2}-1}_{p_2,1})}+\bigl(\|P\|_{L_t^\infty(L^\infty)}
+\|\nabla P\|_{L^\infty_t(B^{d/{p_1}-1}_{p_1,1})}\bigr)\|\nabla Q\|_{L^1_t(B^{d/{p_2}-1}_{p_2,1})}.
\end{split}\end{equation}

So bounding $\nabla Q$ is our next task. 
First of all,  using the fact that $Q$ satisfies the elliptic equation
\eqref{lineareq:Q}, we still have 
$$
m\|\nabla Q\|_{L^2}\leq \|L\|_{L^2}
\quad\hbox{with}\quad L=h+u\cdot\nabla q+\nabla u\cdot \nabla\mu.
$$
Hence, given that $L^2\hookrightarrow B^{d/p_2-d/2}_{p_2,2}$ (here comes that $p_2\geq 2$),
we deduce that
\begin{equation}\label{linearest:Q,p,L2}
m\|\nabla Q\|_{B^{d/{p_2}-d/2}_{p_2,2}}\lesssim  \|L\|_{L^2}.
\end{equation}
Of course, this implies  that 
\begin{align}\label{linearest:Q_-1,p,3D}
m\|\nabla \Delta_{-1}Q\|_{L^{p_2}}\lesssim \|L\|_{L^2}.
\end{align}

In order to bound $\nabla Q$ in $B^{d/p_2-1}_{p_2,1},$ we use again the fact that 
 $$\div(P\nabla Q_j)=\nabla\cdot L_j+\nabla\cdot [P,\dj]\nabla Q.$$
Therefore, 
$$
m\int|Q_j|^{p_2-2}|\nabla Q_j|^2\,dx\lesssim \Int |Q_j|^{p_2-2}|\nabla Q_j\cdot(L_j+[P,\dj]\nabla Q)|\,dx.
$$
Taking advantage of  (\ref{equality:bernstein}), we get after a few computations: 
$$
\|\nabla Q_j\|_{L^{p_2}}
\lesssim  \|L_j\|_{L^{p_2}}+\|[P,\dj]\nabla Q\|_{L^{p_2}}\quad\hbox{for }\  j\geq 0.
$$

 Applying Inequality \eqref{est:com1} with
$$
``u"=P,\quad ``v"=\nabla Q,\quad ``p_1"=p_2,\quad ``p_2"=p_1,\quad ``\nu"=1/4,
\quad ``s"=d/{p_2}-1,
$$
(which is possible provided $1/{p_1}+1/{p_2}>1/d$ and $1/{p_2}\leq 1/{p_1}+5/(4d)$), 
we get 
\begin{equation}\label{linearest:P,Q,p,commute}
\|[P,\dj]\nabla Q\|_{L^{p_2}}\lesssim 2^{-j(d/{p_2}-1)}c_j\|\nabla P\|_{B^{d/{p_1}-3/4}_{p_1,1}}
\|\nabla Q\|_{B^{d/{p_2}-5/4}_{p_2,1}}\quad\hbox{with}\quad
\sum_j c_j=1.
\end{equation}
In the case $d\geq3,$ arguing by interpolation, we get
 \begin{equation}\label{linearest:interpo}
 \|\nabla Q\|_{B^{d/{p_2}-5/4}_{p_2,1}}\lesssim
\|\nabla Q\|_{B^{d/{p_2}-1}_{p_2,1}}^{\frac{2d-5}{2d-4}}\|\nabla Q\|_{B^{d/{p_2}-d/2}_{p_2,2}}^{\frac{1}{2d-4}}.\end{equation}

Therefore together with \eqref{linearest:Q,p,L2} and (\ref{linearest:Q_-1,p,3D}),  
this implies that
\begin{equation}\label{linearest:Q,p,l,P}
m\|\nabla Q\|_{B^{d/{p_2}-1}_{p_2,1}}\lesssim
\|L\|_{B^{d/{p_2}-1}_{p_2,1}}+\bigl(1+\|\nabla P\|_{B^{d/{p_1}-3/4}_{p_1,1}}\bigr)^{2d-4}\|L\|_{L^2}.
\end{equation}

In the case $d=2,$ the interpolation inequality \eqref{linearest:interpo} fails.
However, from (\ref{linearest:Q,p,L2}) and (\ref{linearest:P,Q,p,commute}) we directly get for $p_1,p_2$ satisfying (\ref{condition:p,r}),
\begin{equation}\label{linearest:Q,p,2D}\begin{split}
m\|\nabla Q\|_{B^{2/{p_2}-1}_{p_2,1}}
&\lesssim \|L\|_{B^{2/{p_2}-1}_{p_2,1}}
+\|\nabla P\|_{B^{2/{p_1}-3/4}_{p_1,1}}
\|\nabla Q\|_{B^{2/{p_2}-5/4}_{p_2,1}}\\
&\lesssim \|L\|_{B^{2/{p_2}-1}_{p_2,1}}
+(1+\|\nabla P\|_{B^{2/{p_1}-3/4}_{p_1,1}})
\|L\|_{L^2}.
\end{split}\end{equation}

Therefore, in any dimension $d\geq2,$ we have
\begin{equation}\label{eq:Q}
m\|\nabla Q\|_{B^{d/{p_2}-1}_{p_2,1}\cap L^2}
\leq C_{Q,d,p_1,p_2} \bigl(1+\|\nabla P\|_{B^{d/{p_1}-3/4}_{p_1,1}}\bigr)^{\max(1,2d-4)}
\|L\|_{B^{d/{p_2}-1}_{p_2,1}\cap L^2}.
\end{equation}

In order to treat the case where $\nabla P$ is only in $L^\infty_T(B^{d/p_1-1}_{p_1,1}),$
we proceed exactly as 
in the proof of Proposition  \ref{p:momentum}, decomposing  $P$ into two parts, the smooth large part $S_N P$ and the small rough part $P-S_NP.$ 
Under the same assumptions as in  (\ref{linearest:P,Q,p}), we find that
$$
\|E_N\|_{B^{d/{p_2}-1}_{p_2,1}}\leq C_{P,p_1,p_2}\|P-P_N\|_{B^{d/{p_1}}_{p_1,1}}\|\nabla Q\|_{B^{d/{p_2}-1}_{p_2,1}}.
$$
 Therefore, if $c_{P,p_1,p_2}$ is small enough in (\ref{condition:N,p}), we get 
\begin{equation}\label{linearest:Q,p,l}
\|\nabla Q\|_{L_t^1(B^{d/{p_2}-1}_{p_2,1}\cap L^2)}\leq
 C_{Q,d,p_1,p_2} 2^{N/4}\bigl(1+\|\nabla P\|_{L^\infty_T(B^{d/{p_1}-1}_{p_1,1})}\bigr)^{\max(1,2d-4)}
\|L\|_{L^1_T(B^{d/{p_2}-1}_{p_2,1}\cap L^2)}.
\end{equation}

In order to complete the proof of \eqref{linearest:Q,p}, we now have to bound $L.$
First, we notice that, applying  
 (\ref{est:product Besov1})  with
$``p_1"=``p_2"=p_2,$ $``s_1"=d/{p_2}$ and $``s_2"=d/{p_2}-1$ 
 yields, if $p_2<2d,$ 
\begin{equation}\label{linearest:q,u,p}
\|u\cdot\nabla  q\|_{B^{d/{p_2}-1}_{p_2,1}}\lesssim \|\nabla q\|_{B^{d/{p_2}}_{p_2,1}}
\|u\|_{B^{d/{p_2-1}}_{p_2,1}}.
\end{equation}
If $2d\leq p_2<\infty$ then, owing to $\div u=0,$ the same inequality 
is true. Indeed  applying Bony's decomposition, we discover that
$$
u\cdot\nabla q=T_u\nabla q+T_{\nabla q}u+\div R(u,q-\Delta_{-1}q)+ R(u,\Delta_{-1}\nabla q).
$$
The first two terms may be bounded as in \eqref{linearest:q,u,p}.
For the third one, one has (because $\cF(q-\Delta_{-1}q)$ is supported away from the origin)
$$\begin{array}{lll}
\|\div R(u,q-\Delta_{-1}q)\|_{{B^{d/{p_2}-1}_{p_2,1}}}&\lesssim&
\| R(u,q-\Delta_{-1}q)\|_{{B^{d/{p_2}}_{p_2,1}}}\\
&\lesssim&  \| u\|_{B^{d/{p_2}-1}_{p_2,1}}
\|q-\Delta_{-1}q\|_{B^{d/{p_2+1}}_{p_2,1}}
\lesssim   \| u\|_{B^{d/{p_2}-1}_{p_2,1}}
\|\nabla q\|_{{B^{d/{p_2}}_{p_2,1}}}.\end{array}
$$
And finally, we have 
$$
R(u,\Delta_{-1}\nabla q)=\sum_{-1\leq j\leq0} \Delta_j\Delta_{-1}\nabla q \:
(\Delta_{j-1}\!+\!\Delta j\!+\!\Delta_{j+1})u,
$$
so it is clear that we have
$$
\|R(u,\Delta_{-1}\nabla q)\|_{B^{d/{p_2}-1}_{p_2,1}}\lesssim
\|R(u,\Delta_{-1}\nabla q)\|_{L^{p_2}}
\lesssim \|\nabla q\|_{L^\infty}\|S_2u\|_{L^{p_2}}
\lesssim \|\nabla q\|_{B^{d/{p_2}}_{p_2,1}}
\|u\|_{B^{d/{p_2-1}}_{p_2,1}}.
$$
Next, just as in (\ref{linearest:P,Q,p}), under the conditions $1/p_2\leq1/p_1+1/d$ and 
$1/p_1+1/p_2>1/d,$ we have
\begin{equation}\label{linearest:P,u,p}
\|\nabla u\cdot\nabla \mu\|_{B^{d/{p_2}-1}_{p_2,1}}\lesssim \|\nabla \mu\|_{B^{d/{p_1}}_{p_1,1}}\|\nabla u\|_{B^{d/{p_2}-1}_{p_2,1}}. 
\end{equation}
This gives \eqref{linearest:Q,p}.
\medbreak
In order to complete the proof of the lemma, 
we still have to bound $u\cdot\nabla q$ and $\nabla u\cdot\nabla \mu$
 in $L^2.$
 To handle the former term, we just 
 use the fact that
 $$
 B^{d/p_2-1/2}_{p_2,1}\hookrightarrow L^4\quad\hbox{if}\quad p_2\leq4.
$$ 
Hence, by virtue of H\"older's inequality,
\begin{equation}\label{linearest:q,u,p,L2}
\|u\cdot \nabla q\|_{L^2}\lesssim \|u\|_{B^{d/{p_2}-1/2}_{p_2,1}}\|\nabla q\|_{B^{d/{p_2}-1/2}_{p_2,1}}.\end{equation}

Concerning the latter term, if both $p_1$ and $p_2$ are less than or equal 
to $4$ then one may merely use the embedding
$$
 B^{d/p_1-1/2}_{p_1,1}\hookrightarrow L^4\quad\hbox{and}\quad
  B^{d/p_2-1/2}_{p_2,1}\hookrightarrow L^4,
$$ 
hence
$$
\|\nabla u\cdot\nabla \mu\|_{L^2}\lesssim \|\nabla \mu\|_{B^{d/{p_1}-1/2}_{p_1,1}}\|\nabla u\|_{B^{d/{p_2}-1/2}_{p_2,1}}
$$
Now, if $p_1>4$  then we first write 
$$
\|\nabla u\cdot\nabla \mu\|_{L^2}\leq \|\nabla\mu\|_{L^{p_1}}
\|\nabla u\|_{L^{\tilde p_1}}\quad\hbox{with }\ \tilde p_1=\frac{2p_1}{p_1-2}\cdotp
$$
Let $\eta=\min(1/2,d/p_1).$ 
Then we notice that if $p_2\leq\tilde p_1$ then 
$$
B^{d/p_1-\eta}_{p_1,1}\hookrightarrow L^{p_1}\quad\hbox{and}\quad
B^{d/p_2-1+\eta}_{p_2,1}\hookrightarrow L^{\tilde p_1}.
$$
Therefore 
$$
\|\nabla u\cdot\nabla \mu\|_{L^2}\lesssim \|\nabla \mu\|_{B^{d/{p_1}-\eta}_{p_1,1}}\|\nabla u\|_{B^{d/{p_2}-1+\eta}_{p_2,1}}.
$$

Together with (\ref{linearest:u,p,Q}), interpolation inequalities  and Gronwall lemma, 
this enables us to complete the proof of (\ref{linearest:momentum,p}).
\end{proof}

\begin{rem}\label{linearrem:small time}
The quantities $\tilde C_{1,p_1}\|\Delta_{-1}\theta\|_{L^1_t(L^{p_1})}$, $\tilde C_2\|\Delta_{-1}u\|_{L^1_t(L^2)}$ and  $\tilde C_{2,p_2}\|\Delta_{-1}u\|_{L^1_t(L^{p_2})}$ in the a priori estimates \eqref{linearest:energy,p}, \eqref{linearest:momentum} and \eqref{linearest:momentum,p} respectively can be absorbed if the time $t$ is small. Indeed,  for instance, 
one has for any $s\in\R,$ 
$$
\|\Delta_{-1}\theta\|_{L^1_t(L^{p_1})}
\lesssim \|\theta\|_{L^1_t(B^{s}_{{p_1},1})}
\leq t \|\theta\|_{L^\infty_t(B^{s}_{{p_1},1})},
$$
hence $\tilde C_{1,p_1}\|\Delta_{-1}\theta\|_{L^1_t(L^{p_1})}$ can be absorbed
 by the left-hand side if $t$ is small.
 
 In the case of the linearized momentum equation, 
  plugging \eqref{linearest:momentum} in \eqref{linearest:Q,p}, we thus deduce that, for
  small enough time, one has for some constant $C$ depending only on $d,p_1,p_2,m,M,N,$
  $$\displaylines{
  \|\nabla Q\|_{L^1_t(B^{\frac{d}{p_2}-1}_{p_2,1}\cap L^2)}
\leq C\|h\|_{L_t^1(L^2\cap B^{d/{p_2}-1}_{p_2,1})} \hfill\cr\hfill
+\biggl(\|u_0\|_{B^{d/p_2}_{p_2,1}}
+\|h\|_{L_t^1(L^2\cap B^{d/{p_2}-1}_{p_2,1})}\biggr)
\biggl(e^{C\int_0^t(\|\nabla q\|_{B^{d/p_2}_{p_2,1}}+
\|\nabla q\|_{B^{d/p_2-1/2}_{p_2,1}}^{4/3}+\|\nabla\mu\|_{B^{d/p_1}_{p_1,1}}^2
+\|\nabla\mu\|_{B^{d/p_1-\eta}_{p_1,1}}^{2/(1-\eta)})\,d\tau}-1\biggr).}   
  $$
  \end{rem}

\begin{rem}
Compared to the statement of Proposition \ref{th:Stokes} in the case
$s=d/p_2-1$ one has to assume in addition that $p_2\leq 4$
and also that $p_2\leq 2p_1/(p_1-2)$ (if $p_1\geq2$). This is due to the fact
that bounding $\nabla Q,$ through the elliptic 
equation \eqref{lineareq:Q} requires  a $L^2$ information 
over the right-hand side, that is on $h$ and on quadratic terms.
   The naive idea is just
that, according to H\"older's inequality 
 $L^4$ bounds over  $\nabla q,$ $u,$ $\nabla u$ and $\nabla\mu$
provides this $L^2$ bound. This is the key  to go beyond the energy framework
for \eqref{lineareq:momentum}.  At the same time, we do not know how to treat the case $p_2>4.$
\end{rem}


\subsection{The proof of the well-posedness in the fully nonhomogeneous  case}

We follow the same procedure as in the proof of Theorem \ref{th:global}: 
first we  construct a sequence of  approximate solutions, then we
prove uniform bounds for this sequence and finally, we show the convergence to some  solution
of \eqref{equation:local}. Compared to the almost homogeneous case, 
the main difference is that our estimates rely mostly on Propositions \ref{p:energy}
and \ref{p:momentum,p}. Furthermore, in order to  handle large data, 
we will  have to introduce  the ``free solution'' 
 $(\theta_L,u_L)$  corresponding to data $(\theta_0,u_0),$ namely the solution to 
\begin{equation}\label{eq:libre}
\left\{\begin{array}{ccc}
\d_t\theta_L-\bar\kappa\Delta\theta_L&=&0,\\
\d_t u_L-\bar\mu\Delta u_L&=&0,\\
(\theta_L,u_L)|_{t=0}&=&(\theta_0,u_0),
\end{array}\right.
\end{equation}
with  $\bar\kappa=\kappa(1)$ and $\bar\mu=\mu(1).$

\subsubsection*{Step 1. Construction of a sequence of approximate solutions}

As $\theta_0$ is in $B^{d/p_1}_{p_1,1}$ and $u_0,$ in $B^{d/p_2-1}_{p_2,1},$
the above System \eqref{eq:libre} has a unique global solution 
$(\theta_L,u_L)$ with (see e.g. \cite{BCD})
$$
\theta_L\in\tilde C_T(B^{d/p_1}_{p_1,1})\cap L_T^1(B^{d/p_1+2}_{p_1,1})
\quad\hbox{and}\quad
u_L\in\tilde C_T(B^{d/p_2-1}_{p_2,1})\cap L_T^1(B^{d/p_2+1}_{p_2,1})\quad\hbox{for all }\ T>0,
$$
and we have (if $T$ is small enough and 
with $C$ depending only on $d,p_1,p_2$)
\begin{equation}\label{eq:local,libre}\begin{array}{lll}
\|\theta_L\|_{\tilde L_T^\infty(B^{d/p_1}_{p_1,1})}
+\bar\kappa\|\theta_L\|_{L_T^1(B^{d/p_1+2}_{p_1,1})} &\!\!\!\leq\!\!\!&C \|\theta_0\|_{B^{d/p_1}_{p_1,1}},\\[1ex]
\|u_L\|_{\tilde L_T^\infty(B^{d/p_2-1}_{p_2,1})}
+\bar\mu\|u_L\|_{L_T^1(B^{d/p_2+1}_{p_2,1})}&\!\!\!\leq\!\!\!& C\|u_0\|_{B^{d/p_2-1}_{p_2,1}}.
\end{array}
\end{equation}

Note also that the divergence free property for the initial velocity is
conserved during the evolution. 
Another important feature  is that, owing to $\theta_L\in\tilde L_T^\infty(B^{d/p_1}_{p_1,1}),$
we have, for any $T>0,$
\begin{equation}\label{eq:local,N}
\lim_{N\rightarrow+\infty} \|\theta_L-S_N\theta_L\|_{L^\infty_T(B^{d/p_1}_{p_1,1})}=0.
\end{equation}
Let us fix some small enough positive time $T.$
Given \eqref{eq:local,N}, we see that for any positive constant~$c,$ there exists some
 positive integer $N_0$  so that
\begin{equation}\label{eq:local,0}
\|\theta_L-S_{N_0}\theta_L\|_{L_T^\infty(B^{d/p_1}_{p_1,1})} \leq cm.
\end{equation}
In addition, if the data satisfy \eqref{cond:local initial} then one may assume that
we have (changing $N_0$ and $C$ if need be)
\begin{equation}\label{eq:local,0b}
\frac m2\leq   S_n\vartheta_0\leq CM, \quad \|S_n\theta_0\|_{ B^{d/p_1}_{p_1,1}}+ \|  S_n u_0\|_{  B^{d/p_2-1}_{p_2,1}}\leq CM\quad\textrm{for all}\quad n\geq N_0.
\end{equation}

In order to define our approximate solutions, we use the following iterative scheme:
first we set $(\theta^{0},u^{0},\nabla Q^0)=(  S_{N_0}\theta_0,  S_{N_0}u_0,0)$
(this is obviously a smooth stationary function with  decay at infinity)
then, assuming that the approximate solution $(\theta^n,u^n,\nabla Q^n)$ 
has been constructed over $\R^+\times\R^d,$ 
 we set  $\vartheta^n=1+\theta^n$ and define 
 $(\theta^{n+1},u^{n+1},\nabla Q^{n+1})$ to be the unique solution of the system
\begin{equation*}
\left\{\begin{array}{ccc}
\d_{t}\theta^{n+1}+u^n\cdot\nabla\theta^{n+1} -\div(\kappa^n\nabla\theta^{n+1})&=&f^n,\\
\d_{t}u^{n+1}+u^{n}\cdot\nabla u^{n+1}-\div(\mu^n\nabla u^{n+1})+\vartheta^n\nabla Q^{n+1}&=&h^n,\\
\div u^{n+1}&=&0,\\
(\theta^{n+1},u^{n+1})|_{t=0}&=&(  S_{N_0+n+1}\theta_0,  S_{N_0+n+1}u_0),
\end{array}\right.
\end{equation*}
where
$$
\kappa^n=\lambda k(\vartheta^n),
\quad \mu^n=\beta\zeta(\vartheta^n),
\quad f^n=f(1+\theta^n,u^n)\ \hbox{ and }\
 h^n=h(1+\theta^n,u^n).
$$

Note that the existence and uniqueness of a global smooth solution for the above
system is ensured by the standard theory of parabolic equations
(concerning $\theta^{n+1}$) and by (a slight modification of) Theorem 2.10 in \cite{Alazard06}
(concerning $u^{n+1}$) whenever $(\theta^n,u^n)$ is suitably 
smooth and the coefficients $\kappa^n,\mu^n$ are bounded
by above and by below. 
In fact, given \eqref{eq:local,0b}, the maximum principle ensures that
\begin{equation}\label{eq:ellipticity}
 m/2\leq \vartheta^n\leq CM.
\end{equation}
Hence $\kappa^n$ and $\mu^n$ are bounded by above and from  below
independently of $n.$ 

Next, we notice that if we set 
\begin{equation*}
(\theta^{n}_L,u^{n}_L,\nabla Q^{n}_L)=(  S_{N_0+n}\theta_L, S_{N_0+n}u_L,0)
\end{equation*}
then the equation for $(\bar{\theta}^{n+1},\bar{u}^{n+1},\nabla\bar{Q}^{n+1})
=(\theta^{n+1}-\theta^{n+1}_L,u^{n+1}-u^{n+1}_L,\nabla Q^{n+1}-0)$ reads
\begin{equation*}\left\{
\begin{array}{ccc}
\d_t\bar{\theta}^{n+1}+u^n\cdot\nabla\bar{\theta}^{n+1}
-\div(\kappa^n\nabla\bar{\theta}^{n+1})&=&
F^n,\\
\d_{t}\bar{u}^{n+1}+u^n\cdot\nabla\bar{u}^{n+1}
-\div(\mu^n\nabla\bar{u}^{n+1})
+(1+\theta^n_L)\nabla\bar{Q}^{n+1}&=&
H^n,\\
\div\bar{u}^{n+1}&=&0,\\
(\bar{\theta}^{n+1},\bar{u}^{n+1})|_{t=0}&=&(0,0),
\end{array}
\right.\end{equation*}

where
$$\displaylines{
F^n=-u^n\cdot\nabla\theta_L^{n+1}
+\div((\kappa^n-\bar\kappa)\nabla\theta_L^{n+1})
+f^n,\cr
H^n=-u^n\cdot\nabla u_L^{n+1}
+\div((\mu^n-\bar\mu)\nabla u_L^{n+1})
-\bar\theta^n\nabla\bar Q^{n+1}+h^n.}
$$

Let us point out that, given \eqref{eq:local,0}, we have (up to a harmless
change of $c$)
\begin{equation}\label{eq:local,N_0}
\|\theta_L^n-S_{N_0}\theta_L^n\|_{L_T^\infty(B^{d/p_1}_{p_1,1})} \leq cm
\quad\hbox{for all }\ n\in\N.
\end{equation}


\subsubsection*{Step 2. Uniform bounds}

Bounding $(\bar\theta^{n+1},\bar u^{n+1},\nabla\bar Q^{n+1})$
in terms of the free solution $(\theta_L,u_L)$ and of 
 $(\bar\theta^{n},\bar u^{n},\nabla\bar Q^{n})$
relies on  Propositions   \ref{p:energy} and  \ref{p:momentum,p}
with $s=d/p_1$ and  $s=d/p_2-1,$ respectively, and $N=N_0$
(here we have to take $c$ small enough in \eqref{eq:local,N_0}). 
Using the fact that, as pointed out by Remark \ref{linearrem:small time}, 
taking $T$ smaller if needed allows to discard 
$\tilde C_1\|\Delta_{-1}\bar\theta^{n+1}\|_{L^1_t(L^{p_1})}$  and 
$\tilde C_2\|\Delta_{-1}\bar u^{n+1}\|_{L^1_t(L^{p_2})}$ in the estimates, 
we get,  under the condition (\ref{cond:local,p,r}),
$$
\displaylines{\|\bar\theta^{n+1}\|_{X^{p_1}(T)}\leq
C e^{C_1(\|\nabla u^n\|_{L^1_T(  B^{d/p_2}_{p_2,1})}
+\|\nabla \kappa^n\|^2_{L^2_T( B^{d/p_1}_{p_1,1})})}
\|F^n\|_{L^1_T( B^{d/p_1}_{p_1,1})},\cr
\|\bar u^{n+1}\|_{Y^{p_2}(T)} 
+\|\nabla\bar Q^{n+1}\|_{Z^{p_2}(T)}\leq C\|H^n\|_{L^1_T(B^{d/p_2-1}_{p_2,1}\cap L^2)}\hfill\cr\hfill\times
e^{C_2(\|\nabla u^n\|_{L^1_T( {B}^{d/p_2}_{p_2,1})}
+\|\nabla u^n\|_{L^{4/3}_T( {B}^{d/p_2-1/2}_{p_2,1})}^{4/3}
+\|\nabla\mu^n\|^2_{L^2_T( {B}^{d/p_1}_{p_1,1})}
+\|\nabla\mu^n\|^{4}_{L^{4}_T( {B}^{d/p_1-1/2}_{p_1,1})})}.}
$$
{}From Proposition \ref{prop:action Besov} and elementary interpolation inequalities,
we gather that all the terms in the exponential may be bounded by
$\|\theta^n\|_{X^{p_1}(T)}+\|u^n\|_{Y^{p_2}(T)}$ to some power. 
Therefore, if we assume that 
\begin{equation}\label{eq:local,bound}
\|\theta^n\|_{X^{p_1}(T)}+\|u^n\|_{Y^{p_2}(T)}\leq 2CM
\end{equation}
then we have\footnote{In all that follows, 
we denote by $C_M$ a suitable increasing function of $M.$
To simplify the notation, we omit the dependency with respect to 
$d,$ $N,$ $p_1,$ $p_2,$ etc.}
\begin{equation}\label{eq:local,theta}
\|\bar\theta^{n+1}\|_{X^{p_1}(T)}\leq C_M\Bigl(\|f^n\|_{L^1_T( B^{d/p_1}_{p_1,1})}
+\|u^n\cdot\nabla\theta_L^{n+1}\|_{L^1_T( B^{d/p_1}_{p_1,1})}
+\|(\kappa^n-\bar\kappa)\nabla\theta_L^{n+1}\|_{L^1_T( B^{d/p_1+1}_{p_1,1})}\Bigr),
\end{equation}
\begin{align}\label{eq:local,u}
&\|\bar u^{n+1}\|_{Y^{p_2}(T)}+\|\nabla\bar Q^{n+1}\|_{Z^{p_2}(T)}
\leq C_M\Bigl(\|h^n\|_{L^1_T( {B}^{d/p_2-1}_{p_2,1}\cap L^2)}+\|u^n\cdot\nabla u_L^{n\!+\!1}\|_{L^1_T( {B}^{d/p_2-1}_{p_2,1}\cap L^2})\nonumber\\
&\qquad\qquad\qquad+\|\div((\mu^n-\bar\mu)\nabla u_L^{n+1})\|_{L^1_T( {B}^{d/p_2-1}_{p_2,1}\cap L^2)}
+\|\bar\theta^n\nabla\bar Q^{n\!+\!1}\|_{L^1_T( {B}^{d/p_2-1}_{p_2,1}\cap L^2)}\Bigr).\qquad\quad
\end{align}

So bounding the right-hand sides of \eqref{eq:local,theta} and of \eqref{eq:local,u}
is our next task.
Given \eqref{eq:local,bound}, we easily get  from Propositions \ref{prop:product Besov} and \ref{prop:action Besov}:
$$\begin{array}{rcl}
\|f^n\|_{L^1_T( B^{d/p_1}_{p_1,1})}&\!\!\!\leq\!\!\!& 
C_M \|\nabla\theta^n\|_{L_T^2(B^{d/p_1}_{p_1,1})}^2,\\[1.5ex]
\|(\kappa^n-\bar\kappa)\nabla\theta_L^{n+1}\|_{L^1_T( B^{d/p_1+1}_{p_1,1})}
&\!\!\!\leq\!\!\!& C_M\bigl(\|\theta^n\|_{L^\infty_T(B^{d/p_1}_{p_1,1})}
\|\nabla\theta_L\|_{L^1_T(B^{d/p_1+1}_{p_1,1})}\\
&&\qquad\qquad\qquad+\|\nabla\theta_L\|_{L^2_T(B^{d/p_1}_{p_1,1})}\|\theta^n\|_{L^2_T(B^{d/p_1+1}_{p_1,1})}\bigr).
\end{array}
$$
If  $p_2\leq p_1,$  the space $B^{d/p_2}_{p_2,1}$ is embedded
in  the Banach algebra $B^{d/p_1}_{p_1,1}.$ Hence
\begin{equation}\label{eq:local,2}
\|u^n\cdot\nabla\theta_L^{n+1}\|_{L^1_T( B^{d/p_1}_{p_1,1})}
\leq C\|u^n\|_{L^2_T(B^{d/p_2}_{p_2,1})}\|\nabla\theta_L\|_{L^2_T(B^{d/p_1}_{p_1,1})}.
\end{equation}
If  $p_1<p_2$ then \eqref{eq:local,2} is no longer true. 
 However, from Bony's decomposition  and  
 Proposition \ref{p:paraproduct}, it is not difficult to get that
 $$
 \|u^n\cdot\nabla\theta_L^{n+1}\|_{L^1_T( B^{d/p_1}_{p_1,1})}
\lesssim \|u^n\|_{L^2_T(B^{d/p_2}_{p_2,1})}\|\nabla\theta_L\|_{L^2_T(B^{d/p_1}_{p_1,1})}
+ \|u^n\|_{L^{2/(1\!-\!\varepsilon)}_T(B^{d/p_2-\varepsilon}_{p_2,1})}\|\nabla\theta_L\|_{L^{2/(1+\varepsilon)}_T(B^{d/p_1+\varepsilon}_{p_1,1})}
$$ 
 whenever $\varepsilon\in[0,1]$ and $d/p_1\leq1-\varepsilon+d/p_2.$
\smallbreak

Next, computations similar to those that enable us to bound $h^n$
in the homogeneous framework lead to 
\begin{align*}
&\|h^n\|_{L^1_T( {B}^{d/p_2-1}_{p_2,1})}\leq C_M\Bigl(
\|\nabla\theta^n\|_{L_T^2(B^{d/p_1}_{p_1,1})}^2+
\|\nabla^2\theta^n\|_{L_T^2(B^{d/p_1-1}_{p_1,1})}
\|\nabla\theta^n\|_{L_T^2(B^{d/p_1}_{p_1,1})}\\
&\qquad\qquad\quad+\|\nabla^2\theta^n\|_{L^{{2}/{(2-\varepsilon)}}_T(B^{d/p_1-\varepsilon}_{p_1,1})}\|\nabla\theta^n\|_{L^{{2}/{\varepsilon}}_T(B^{d/p_1-1+\varepsilon}_{p_1,1})}+\|\nabla\theta^n\|_{L_T^2(B^{d/p_1}_{p_1,1})}\|\nabla u^n\|_{L_T^2(B^{d/p_2-1}_{p_2,1})}
\Bigr),
\end{align*}
provided   
$$p_1<2d, \quad p_1\leq 2p_2,\quad
 \frac 1{p_1}+\frac{1}{p_2}>\frac{1}{d}\ \hbox{ and }\ 
 \frac 1{p_2}\leq \frac{1}{p_1}+\frac{1}{d}-\frac{\varepsilon}{d}\ \hbox{ for some }\   
\varepsilon\in[0,1],$$
and for $p_1\leq 4$,
\begin{align*}
&
\|h^n\|_{L^1_T(L^2)}\leq C_M\Bigl(\|\nabla\theta^n\|_{L^2_T(L^\infty)}\|\nabla\theta^n\|_{L^4_T(B^{d/p_1-1/2}_{p_1,1})}^2\\
&\qquad\qquad+\|\nabla^2\theta^n\|_{L^{4/3}_T(B^{d/p_1-1/2}_{p_1,1})}\|\nabla\theta^n\|_{L^{4}_T(B^{d/p_1-1/2}_{p_1,1})}+\|\nabla\theta^n\|_{L^4_T(B^{d/p_1-1/2}_{p_1,1})}\|\nabla u^n\|_{L^{4/3}_T(B^{d/p_2-1/2}_{p_2,1})}\Bigr).
\end{align*}

Under Condition \eqref{cond:local,p,r},
 H\"older inequality, Propositions \ref{prop:product Besov} and \ref{prop:action Besov} also give
\begin{align*}
&\|u^n\!\cdot\!\nabla u_L^{n+1}\|_{L^1_T( {B}^{d/p_2\!-\!1}_{p_2,1}\cap L^2)}
\lesssim\|u^n\|_{L^2_T(B^{d/p_2}_{p_2,1})}
\|u_L\|_{L^{2}_T( {B}^{d/p_2}_{p_2,1})}\!+\!\|u^n\|_{L^4_T(B^{d/p_2\!-\!1/2}_{p_2,1})}
\|\nabla u_L\|_{L^{4/3}_T( {B}^{d/p_2-1/2}_{p_2,1})}\\[1ex]
&\|\div\bigl((\mu^n\!-\!\bar\mu)\nabla u_L^{n\!+\!1}\bigr)\|_{L^1_T( {B}^{d/p_2-1}_{p_2,1}\cap L^2)}
\lesssim \|\nabla\theta^n\|_{L^2_T(B^{d/p_1}_{p_1,1})}
\|\nabla u_L\|_{L^2_T( {B}^{d/p_2-1}_{p_2,1})}\\
&\qquad\qquad\qquad\qquad+\|\theta^n\|_{L^{\infty}_T(B^{d/p_1}_{p_1,1})}
\|u_L\|_{L^1_T( {B}^{d/p_2+1}_{p_2,1})}  +\|\nabla\theta^n\|_{L^4_T(B^{d/p_1-1/2}_{p_1,1})}
\|\nabla u_L\|_{L^{4/3}_T( {B}^{d/p_2-1/2}_{p_2,1})},\\[1ex]
&\|\bar\theta^n\nabla\bar Q^{n+1}\|_{L^1_T( {B}^{d/p_2-1}_{p_2,1}\cap L^2)}
\lesssim
 \|\bar\theta^n\|_{L^\infty_T(B^{d/p_1}_{p_1,1})}
\|\nabla\bar Q^{n+1}\|_{L^1_T( {B}^{d/p_2-1}_{p_2,1}\cap L^2)}.
\end{align*}

Let us fix some small positive constant  $\tau_M$
that we shall specify later on and let us assume that $T$ has been chosen so that
\begin{equation}\label{eq:local,free}
\|\theta_L\|_{L^{2/\varepsilon}_T(B^{d/{p_1}+\varepsilon}_{p_1,1})\cap L^{4}_T(B^{d/p_1+1/2}_{p_1,1})\cap L_T^1(B^{d/p_1+2}_{p_1,1})}+\|u_L\|_{L^4_T(B^{d/p_2-1/2}_{p_2,1})\cap L_T^1(B^{d/p_2+1}_{p_2,1})}\leq\tau_M.
\end{equation}

Note that, in order that   the above condition is  satisfied for some positive $T$
even if $\theta_0$ is large,    we have to rule out the 
case $\varepsilon=0.$ 
This accounts for the strict inequality in the conditions
$$
d/p_1<1+d/p_2\quad\hbox{and}\quad d/p_2<1+d/p_1,
$$
that we did not have in the statement of Theorem \ref{th:local}. 

\medbreak

Now,   plugging all the above estimates in \eqref{eq:local,theta} and \eqref{eq:local,u} yields
(up to a harmless change of $C_M$)
\begin{align*}
&\|\bar\theta^{n+1}\|_{X^{p_1}(T)}+\|\bar u^{n+1}\|_{Y^{p_2}(T)}+\|\nabla\bar Q^{n+1}\|_{Z^{p_2}(T)}\\
&\qquad\qquad\qquad\qquad\qquad\qquad\leq 
C_M\Bigl(\|\bar\theta^n\|_{X^{p_1}(T)}(\|\bar\theta^n\|_{X^{p_1}(T)}+\|\bar u^n\|_{Y^{p_2}(T)}+\|\nabla\bar Q^{n+1}\|_{Z^{p_2}(T)})\\
&\qquad\qquad\qquad\qquad\qquad\qquad\qquad\qquad+\tau_M\bigl(\|\bar\theta^{n}\|_{X^{p_1}(T)}+\|\bar u^{n}\|_{Y^{p_2}(T)}\bigr)
+\tau_M^2+\tau_M\|\theta_L\|_{L_T^\infty(B^{d/p_1}_{p_1,1})}\Bigr).
\end{align*}

Using \eqref{eq:local,libre} so as to bound the last term, we see that
 if we   assume that 
\begin{equation}\label{eq:local,bar}
\|\bar\theta^{n}\|_{X^{p_1}(T)}+\|\bar u^{n}\|_{Y^{p_2}(T)}+\|\nabla\bar Q^{n}\|_{Z^{p_2}(T)}\leq K\tau_M
\end{equation}
for some $K=K(M)$ that we shall choose below, 
and take  $\tau_M$ so that
$$
C_M\tau_MK\leq 1/2
$$
then 
$$\|\bar\theta^{n+1}\|_{X^{p_1}(T)}+\|\bar u^{n+1}\|_{Y^{p_2}(T)}+\|\nabla\bar Q^{n+1}\|_{Z^{p_2}(T)}\leq 2C_M\tau_M\bigl(M+(K^2+K+1)\tau_M\bigr).
$$
Hence
$(\bar\theta^{n+1},\bar u^{n+1},\nabla\bar Q^{n+1})$ satisfies \eqref{eq:local,bar} too
if we take $K=2C_M(1+M)$ and assume that $\tau_M$ also satisfies
$$
(1+K+K^2)\tau_M\leq 1.
$$

This completes the proof of a priori estimates on any interval 
$[0,T]$ such that \eqref{eq:local,free} is fulfilled.


\subsubsection*{Step 3. Convergence}

The equation for $(\dt^{n+1},\du^{n+1},\nabla\dQ^{n+1})=
(\theta^{n+1}-\theta^{n},u^{n+1}-u^{n},\nabla Q^{n+1}-\nabla Q^{n})$ reads
\begin{equation*}\left\{\begin{array}{ccc}
\d_t\dt^{n+1}+u^n\cdot\nabla\dt^{n+1}-\div(\kappa^n\nabla\dt^{n+1})
&=&I^n,\\
\d_t\du^{n+1}+u^n\cdot\nabla\du^{n+1}-\div(\mu^n\nabla\du^{n+1})
+(1+\theta^n_L)\nabla\dQ^{n+1}&=&J^n,\\
\div \du^{n+1}&=&0,\\
(\dt^{n+1},\du^{n+1})|_{t=0}&=&( \Delta_{N_0+n}\theta_0, \Delta_{N_0+n} u_0),
\end{array}\right.\end{equation*}
where
\begin{align*}
I^n=&-\du^n\cdot\nabla\theta^n+\div(\delta\!\kappa^n\nabla\theta^n)+f^n-f^{n-1},\\
J^n=&-\du^n\cdot\nabla u^n+\div(\delta\!\mu^n\nabla u^n)-\dt^n\nabla Q^n-\bar\theta^n\nabla\dQ^{n+1}+h^n-h^{n-1}.
\end{align*}

Let $b_n=2^{(N_0+n)d/p_1}\|\Delta_{N_0+n}\theta_0\|_{L^{p_1}}$ and 
$d_n=2^{(N_0+n)(d/p_2-1)}\|\Delta_{N_0+n} u_0\|_{L^{p_2}}.$
Since $\theta_0\in  B^{d/p_1}_{p_1,1}$ and $u_0\in B^{d/p_2-1}_{p_2,1}$, we have
$(b_n)\in\ell^1$ and $(d_n)\in\ell^1.$
\smallbreak

To simplify the presentation, we assume that $p_1\geq p_2.$ Then,
applying Propositions \ref{p:energy} and \ref{p:momentum,p}
and using the bounds of the previous step,  we get (bearing in mind that if $\tau_M$
is sufficiently small in \eqref{eq:local,bar} then one may  
absorb $\bar\theta^n\nabla\dQ^{n+1}$):
$$\displaylines{
\|\dt^{n+1}\|_{X^{p_1}(T)}\leq Cb_n+C_M\Bigl(\bigl(\|\du^n\|_{L^2_T(B^{d/p_2}_{p_2,1})}
+\|\dt^n\|_{L^2_T(B^{d/p_1+1}_{p_1,1})}\bigr)
\|\theta^n\|_{L^2_T(B^{d/p_1+1}_{p_1,1})}\hfill\cr\hfill
+\|\dt^n\|_{L^\infty_T(B^{d/p_1}_{p_1,1})}\bigl(\|\theta^n\|_{L^1_T(B^{d/p_1+2}_{p_1,1})}
+\|\theta^n\|_{L^2_T(B^{d/p_1+1}_{p_1,1})}^2\bigr)
\hfill\cr\hfill+\|\dt^n\|_{L^2_T(B^{d/p_1+1}_{p_1,1})}\|\theta^{n-1}\|_{L^2_T(B^{d/p_1+1}_{p_1,1})}\Bigr),}
$$
$$\displaylines{\|\du^{n+1}\|_{Y^{p_2}(T)}+\|\nabla\dQ^{n+1}\|_{Z^{p_2}(T)}
\leq Cd_n\hfill\cr\hfill+C_M\Bigl(
\|\dt^n\|_{L^\infty_T(B^{d/p_1}_{p_1,1})}\bigl(\|(\theta^{n-1},\theta^n)\|^2_{L^2_T(B^{d/p_1+1}_{p_1,1})}+\|\theta^{n-1}\|_{L^1_T(B^{d/p_1+2}_{p_1,1})}+\|\nabla Q^n\|_{Z^{p_2}(T)}\bigr)
\hfill\cr\hfill
+\|\dt^n\|_{L^2_T(B^{d/p_1+1}_{p_1,1})}\bigl(\|(\theta^n,\theta^{n-1})\|_{L^2_T(B^{d/p_1+1}_{p_1,1})}+\|(\theta^n,\theta^{n-1})\|^2_{L^4_T(B^{d/p_1+1/2}_{p_1,1})}+\|u^n\|_{L^2_T(B^{d/p_2}_{p_2,1})}\bigr)
\hfill\cr\hfill
+\|\dt^n\|_{L^4_T(B^{d/p_1+1/2}_{p_1,1})}\bigl(\|u^n\|_{L^{4/3}_T(B^{d/p_2+1/2}_{p_2,1})}+\|\theta^{n-1}\|_{L^{4/3}_T(B^{d/p_1+3/2}_{p_1,1})}\bigr)
\hfill\cr\hfill+\|\dt^n\|_{L^{4/3}_T(B^{d/p_1+3/2}_{p_1,1})}\|\theta^n\|_{L^4_T(B^{d/p_1+1/2}_{p_1,1})}
+\|\dt^n\|_{L^{2/(2-\varepsilon)}_T(B^{d/p_1+2-\varepsilon}_{p_1,1})}\|\theta^n\|_{L^{2/\varepsilon}_T(B^{d/p_1+\eps}_{p_1,1})}
\hfill\cr\hfill
+\|\du^n\|_{L^2_T(B^{d/p_2}_{p_2,1})}\bigl(\|u^n\|_{L^{2}_T(B^{d/p_2}_{p_2,1})}+\|\theta^{n-1}\|_{L^{2}_T(B^{d/p_1+1}_{p_1,1})}\bigr)
\hfill\cr\hfill +\|\du^n\|_{L^{4}_T(B^{d/p_2-1/2}_{p_2,1})}\|u^n\|_{L^{4/3}_T(B^{d/p_2+1/2}_{p_2,1})}+\|\du^n\|_{L^{4/3}_T(B^{d/p_2+1/2}_{p_2,1})}\|\theta^{n-1}\|_{L^{4}_T(B^{d/p_1+1/2}_{p_1,1})}
\Bigr)}
$$
where $C=C(d,p_1,p_2).$
\smallbreak
Let us emphasize that, according to \eqref{eq:local,free} and \eqref{eq:local,bar}, 
the previous inequalities imply that, up to a change of $C_M,$ we have
$$
B^{n+1}(T)\leq  C_M\tau_MB^n(T)+C(b_n+d_n).
$$
with $B^n(T)=\|\dt^n\|_{X^{p_1}(T)}+\|\du^n\|_{Y^{p_2}(T)}+\|\nabla\dQ^n\|_{Z^{p_2}(T)}.$
\smallbreak
Therefore, taking  $\tau_M$ small enough,  we end up with
$$
B^{n+1}(T)\leq \frac{1}{2}B^n(T)+C(b_{n}+d_n).
$$

As $(b_n)$ and $(d_n)$ are in $\ell^1,$ 
one may thus conclude that $\sum(B^n(T))<\infty$ , which is to say $(\theta^n,u^n,\nabla Q^n)_{n\in\N}$ is a Cauchy sequence and converges to a solution $(\theta,u,\nabla Q)$ of the system (\ref{equation:local}) in the space $F_T^{p_1,p_2}$ which also satisfies the estimates (\ref{estimate:local}) (that $\vartheta\geq m$ is a consequence of the maximum principle
for the parabolic equation satisfied by $\vartheta$). 


\subsubsection*{Step 4. Stability estimates and uniqueness}

To prove the stability, i.e. the continuity of the flow map, and the uniqueness,
 we consider  two solutions $(\theta^1,u^1,\nabla Q^1)$ and $(\theta^2,u^2,\nabla Q^2)$
  of System (\ref{equation:local}) in $F_T^{p_1,p_2}$ with initial data $(\theta^1_0,u^1_0)$ and $(\theta^2_0,u^2_0),$ respectively. 
    Let $\vartheta^1=1+\theta^1$ and $\vartheta^2=1+\theta^2.$
  We assume in addition that
   $$
  \vartheta^1,\vartheta^2\geq m
   $$
   and we fix some large enough integer $N_1$ so that
   $$m/2\leq 1+S_{N_1}\theta^1\quad\hbox{and}\quad
    \|\theta^1-S_{N_1}\theta^1\|_{L^\infty_T(B^{d/p_1}_{p_1,1})}\leq  cm
    $$
        with $c$ given by Condition \eqref{condition:N,p}.
       Finally, we denote by $M$ a common bound for the two 
       solutions in $F_T^{p_1,p_2}.$ 
        
    \smallbreak
    The proof goes from arguments similar to those of the previous step: 
    we notice that the difference of the two solutions
$(\dt,\du,\nabla\dQ):=(\theta^1-\theta^2,u^1-u^2,\nabla Q^1-\nabla Q^2)$ satisfies
  \begin{equation*}\left\{\begin{array}{ccc}
  \d_t\dt+u^1\cdot\nabla\dt-\div(\kappa(\theta^1)\nabla\dt)&=&I,\\
  \d_t\du+u^1\cdot\nabla\du-\div(\mu(\theta^1)\nabla\du)+\vartheta^1\nabla\delta Q&=&J,\\
  \div \du&=&0,\\
  (\dt,\du)|_{t=0}&=&(\dt_0,\du_0),
  \end{array} \right. \end{equation*}
where
\begin{align*}
I=&-\du\cdot\nabla\theta^2+\nabla\cdot(\delta\!\kappa\nabla\theta^2)+f(1+\theta^1)-f(1+\theta^2),\\
J=&-\du\cdot\nabla u^2+\nabla\cdot(\delta\!\mu\nabla u^2)-\dt\nabla Q^2+h(1+\theta^1,u^1)-h(1+\theta^2,u^2).
\end{align*}

Let $B(t)=\|\dt\|_{X^{p_1}(t)}+\|\du\|_{Y^{p_2}(t)}+\|\nabla\dQ\|_{Z^{p_2}(t)}.$ Then arguing exactly 
as in the previous step, we get for small enough $t,$
$$
\displaylines{
B(t)\leq C_{M,N_1}\Bigl(\|\dt_0\|_{B^{d/p_1}_{p_1,1}}+\|\du_0\|_{B^{d/p_2-1}_{p_2,1}}
+\bigl(\|(\theta^1,\theta^2)\|_{L^2_t(B^{d/p_1+1}_{p_1,1})\cap L^4_t(B^{d/p_1+1/2}_{p_1,1})}
\hfill\cr\hfill
+\|\theta^1\|_{L^{4/3}_t(B^{d/p_1+3/2}_{p_1,1})\cap L^{4}_t(B^{d/p_1+1/2}_{p_1,1})}
+\|\theta^2\|_{L^{1}_t(B^{d/p_1+2}_{p_1,1})\cap L^{2/\varepsilon}_t(B^{d/p_1+\varepsilon}_{p_1,1})}
\hfill\cr\hfill
+\|u^2\|_{L^2_t(B^{d/p_2}_{p_2,1})\cap L^{4/3}_t(B^{d/p_2+1/2}_{p_2,1})}
+\|\nabla Q^2\|_{Z^{p_2}(t)}\bigr)B(t)\Bigr).}
$$

If the initial data coincide then we have $B(0)=0.$  Given that the factors of 
$B(t)$ in the right-hand side go to $0$ when $t$ goes to $0,$
we  thus get $B\equiv0$ on a small enough time interval.
Then, from standard continuation arguments, we conclude
to uniqueness  on the whole time interval $[0,T].$

In the more general case where the initial data do not coincide, 
then one may split both solutions into 
$$
\theta^i=\theta^i_L+\bar\theta^i\quad\hbox{and}\quad
u^i=u^i_L+\bar u^i,
$$
where $(\theta^i_L,u^i_L)$ stands for the 
free solution of \eqref{eq:libre} pertaining 
to data $(\theta^i_0,u^i_0).$

If $\|\dt_0\|_{B^{d/p_1}_{p_1,1}}+\|\du_0\|_{B^{d/p_2-1}_{p_2,1}}\leq\delta$, then 
we have
$$
\|\theta_L^2-\theta_L^1\|_{X^{p_1}(T)}+
\|u_L^2-u_L^1\|_{Y^{p_2}(T)}\leq C\delta.
$$

By arguing as in step two, one may also prove that 
if $T$ is so small as to satisfy \eqref{eq:local,free} for (say)
$(\theta_L^1,u_L^1)$  and $\tau_M=\delta$ (with $\delta$ small enough) then 
$(\bar\theta^1,\bar u^1,\nabla\bar Q^1)$ satisfies 
\eqref{eq:local,bar}.
Therefore, from the above inequality, one may conclude that
 $B(T)\leq 2C_{M,N_1}\delta.$
 This completes the proof of  the continuity of the flow map.

\appendix
\section{Appendix}
\setcounter{equation}{0}

For the sake of completeness, we here prove  the commutator estimates
stated in  Proposition \ref{prop:commutator Besov}.
Throughout, it will be understood that   
$\|(c_j)_{j\in \Z}\|_{\ell^{r_2}}=\|(d_j)_{j\in \Z}\|_{\ell^{r_1}}=1$  and that 
\begin{equation}\label{eq:r}
\frac1r=\min\Bigl\{1,\frac{1}{r_1}+\frac{1}{r_2}\Bigr\}\cdotp
\end{equation}

Let $R_j(u,v):=[u,\dj]v$  and $\tilde{u}=u-\Delta_{-1}u.$
 Then we write the decomposition 
$$
R_j(u,v)=R_j^1(u,v)+R_j^2(u,v)+R_j^3(u,v)+R_j^4(u,v)+R_j^5(u,v)
$$
with 
$$
\begin{array}{c}
R_j^1(u,v):=[T_{\tilde  u},\dj]v,\quad
R_j^2(u,v):=T'_{\dj v} {\tilde u},\quad
R_j^3(u,v):= -\dj T_{v}\tilde   u,\\[2ex]
R_j^4(u,v):=-\dj R(\tilde   u,v),\qquad
R_j^5(u,v):=[\Delta_{-1}u,\dj]v. 
\end{array}
$$

Let us first prove inequalities \eqref{est:com1}, \eqref{est:com2} and \eqref{est:com3}. 
By virtue of the first-order Taylor's formula, we have
\begin{align*}
R_j^1(u,v)&=\sum_{|j-j'|\leq 4}[S_{j'-1}\tilde   u,\dj]\Delta_{j'}v\\
&=2^{-j'}\Int_{\R^d}\Int^1_0h(y)y\cdot\nabla S_{j'-1}\tilde   u(x-2^{-j}t'y)\Delta_{j'}v(x-2^{-j}y)\,dt'\,dy,
\end{align*}
hence
$$
\|R^1_j(u,v)\|_{L^{p_1}}\lesssim 2^{-j}\|\nabla S_{j-1}\tilde   u\|_{L^\infty}\|\dj v\|_{L^{p_1}},
$$
whence
$$
\bigl\|(2^{js}\|R^1_j(u,v)\|_{L^{p_1}})_{j\in \N}\bigr\|_{\ell^r}\lesssim\bigl\|(2^{j(\nu-1)}\|\nabla S_{j-1}\tilde u\|_{L^\infty}2^{j(s-\nu)}\|\dj v\|_{L^{p_1}})\bigr\|_{\ell^r}.
$$
So in the case $\nu=1,$ we readily get
\begin{equation}\label{commutorest:R1a}
\|(2^{js}\|R^1_j(u,v)\|_{L^{p_1}})_{j\in \N}\|_{\ell^r}\lesssim
\|\nabla\tilde   u\|_{L^\infty}\|v\|_{B^{s-1}_{p_1,r}}.
\end{equation}
To handle the case $\nu<1,$ we use the fact that 
$$
2^{j(\nu-1)}\|\nabla S_{j-1}\tilde u\|_{L^\infty}2^{j(s-\nu)}\|\dj v\|_{L^{p_1}}
\leq\!\sum_{j'\leq j-2} \!2^{(j-j')(\nu-1)}\bigl(2^{j'(\nu-1)}\|\nabla\Delta_{j'}\tilde u\|_{L^\infty}\bigr)
\bigl(2^{j(s-\nu)}\|\Delta_jv\|_{L^{p_1}}\bigr).
$$
Thus from H\"older and convolution inequalities for series and under assumption \eqref{eq:r}, one may conclude  that 
\begin{equation}\label{commutorest:R1b}
\|(2^{js}\|R^1_j\|_{L^{p_1}})_{j\in \N}\|_{\ell^r}\lesssim
\|\nabla \tilde   u\|_{B^{\nu-1}_{\infty,r_2}}\|v\|_{B^{s-\nu}_{{p_1},r_1}}\ \hbox{ if }\nu<1.
\end{equation}
Concerning $R^2_j(u,v)$, we have
$$
\|R^2_j(u,v)\|_{L^{p_1}}\leq\sum_{j'\geq j-3}\|\Delta_{j'}\tilde u\, S_{j'+2}\dj v\|_{L^{p_1}},
$$
and we consider the following two cases (still under Condition \eqref{eq:r}):
\begin{itemize}
\item $p_1\geq p_2$:
\begin{align*}
\qquad \qquad 2^{js}\|R^2_j(u,v)\|_{L^{p_1}}
&\leq 2^{js}\sum_{j'\geq j-3}\|\Delta_{j'}\tilde u\|_{L^{p_1}}\|S_{j'+2}\,\dj v\|_{L^\infty}\\
&\lesssim 2^{js}\sum_{j'\geq j-3}2^{j'd(\frac{1}{p_2}-\frac{1}{p_1})}\|\Delta_{j'}\tilde u\|_{L^{p_2}}\|\dj v\|_{L^\infty}\\
&\lesssim \sum_{j'\geq j-3}2^{(j-j')(\nu+\frac{d}{p_1})}c_{j'}
\|\tilde  u\|_{B^{\frac{d}{p_2}+\nu}_{p_2,r_2}}d_j\|v\|_{B^{s-\nu-\frac{d}{p_1}}_{\infty,r_1}},
\end{align*}

\item $p_1<p_2$:
\begin{align*}
\qquad \qquad 2^{js}\|R^2_j(u,v)\|_{L^{p_1}}
&\leq 2^{js}\sum_{j'\geq j-3}\|\Delta_{j'}\tilde u\|_{L^{p_2}}\|S_{j'+2}\dj v\|_{L^{\frac{p_1p_2}{p_2-p_1}}}\\
&\lesssim \sum_{j'\geq j-3}2^{(j-j')(\nu+\frac{d}{p_2})}c_{j'}
\|\tilde  u\|_{B^{\frac{d}{p_2}+\nu}_{p_2,r_2}}d_j\|v\|_{B^{s-\nu-\frac{d}{p_2}}_{\frac{p_1p_2}{p_2-p_1},r_1}}.
\end{align*}
\end{itemize}
Hence for $\nu>-d\min\{\frac{1}{p_1},\frac{1}{p_2}\}$,
\begin{equation}\label{commutorest:R2}
\|(2^{js}\|R^2_j(u,v)\|_{L^{p_1}})_{j\in \N}\|_{\ell^r}\lesssim
\|\tilde  u\|_{B^{\frac{d}{p_2}+\nu}_{p_2,r_2}}\|v\|_{B^{s-\nu}_{p_1,r_1}}.
\end{equation}

In the case $p_1\leq p_2,$ bounding $R^3_j(u,v)$ stems from 
\eqref{est:product Besov,paraproduct1}, \eqref{est:product Besov,paraproduct2} (and an obvious embedding
in the limit case). We get 
\begin{equation}\label{commutorest:R3}
\|(2^{js}\|R^3_j(u,v)\|_{L^{p_1}})_{j\in \N}\|_{\ell^r}\lesssim
\left\{\begin{array}{cc}
\|\tilde  u\|_{B^{\frac{d}{p_2}+\nu}_{p_2,r_2}}\|v\|_{B^{s-\nu}_{p_1,r_1}},&\hbox{ if }s<\nu+d/p_2,\\
\|\tilde  u\|_{B^{\frac{d}{p_2}+\nu}_{p_2,r}}\|v\|_{B^{s-\nu}_{p_1,1}},&\hbox{ if }s=\nu+d/p_2.
\end{array}\right.
\end{equation}
To deal with the case $p_1>p_2,$ we just have to notice that, 
according to \eqref{est:product Besov,paraproduct1}, \eqref{est:product Besov,paraproduct2}, the paraproduct operator maps 
$B^{s-\nu-\frac{d}{p_1}}_{\infty,r_1}\times B^{\frac{d}{p_1}+\nu}_{p_1,r_2}$
in $B^s_{p_1,r}$ provided that $s<\nu+d/p_1$ 
(and $L^\infty\times B^{\frac{d}{p_1}+\nu}_{p_1,r}$ 
in $B^s_{p_1,r}$ if $s=\nu+d/p_1$).
So  we still get \eqref{commutorest:R3} provided
$s<\nu+d/p_1$ and $s\leq\nu+d/p_1,$ respectively. 
\smallbreak

As for the fourth term, it is only a matter of applying Inequality 
\eqref{est:product Besov,remainder}. We get 
\begin{equation}\label{commutorest:R4}
\|(2^{js}\|R^4_j(u,v)\|_{L^{p_1}})_{j\in \N}\|_{\ell^r}\lesssim
\|\tilde  u\|_{B^{\frac{d}{p_2}+\nu}_{p_2,r_2}}\|v\|_{B^{s-\nu}_{p_1,r_1}},\hbox{ if } s>-d\min\Bigl\{\frac{1}{p'_1},\frac{1}{p_2}\Bigr\}\cdotp
\end{equation}

The term $R^5_j(u,v)$ may be treated by arguing like in the proof of (\ref{commutorest:R1a}). One ends up with
\begin{equation}\label{commutorest:R5}
\|(2^{js}\|R^5_j(u,v)\|_{L^{p_1}})_{j\in \N}\|_{\ell^r}\lesssim
\|\nabla\Delta_{-1}  u\|_{L^\infty}\|v\|_{B^{s-\nu}_{p_1,r}},\qquad\hbox{ if }\nu\leq  1.
\end{equation}
Given that for any $(s,p,r),$ one has (owing to the low-frequency cut-off)
$$
\|\tilde u\|_{B^s_{p,r}}\lesssim \|\nabla u\|_{B^{s-1}_{p,r}},
$$
putting together 
 (\ref{commutorest:R1a}),  (\ref{commutorest:R1b}),  (\ref{commutorest:R2}), (\ref{commutorest:R3}), (\ref{commutorest:R4}) and (\ref{commutorest:R5}) 
 completes the proof of  \eqref{est:com1}, \eqref{est:com2} and \eqref{est:com3}. 
 \smallbreak
 In order to establish \eqref{est:com4}, we notice that the terms 
 $R_j^i(u,v)$ with $i\not=2$ are spectrally localized in balls of size $2^j.$
 Hence Bernstein inequality together with  (\ref{commutorest:R1a}),  (\ref{commutorest:R1b}),   (\ref{commutorest:R3}), (\ref{commutorest:R4}) and (\ref{commutorest:R5}) 
  ensures that  they satisfy the desired inequality
 under Condition  \eqref{eq:condcom}.
  
 On the other hand $R_j^2(u,v)$ does not have this  spectral localization property. 
 Let us just treat the case $p_1\geq p_2$
 to simplify the presentation. We have
 \begin{equation}\label{est:com5}
 \|\d_kR_j^2(u,v)\|_{L^{p_1}}\leq \sum_{j'\geq j-3}\Bigl(\|\Delta_{j'}\tilde u\|_{L^{p_1}}
 \|\d_kS_{j'+2}\dj v\|_{L^\infty}+
 \|\d_k\Delta_{j'}\tilde u\|_{L^{p_1}}
 \|S_{j'+2}\dj v\|_{L^\infty}\Bigr).
 \end{equation}
According to Bernstein's inequality, we have 
$$  \sum_{j'\geq j-3}\|\Delta_{j'}\tilde u\|_{L^{p_1}}
 \|\d_kS_{j'+2}\dj v\|_{L^\infty}\leq C2^j  \sum_{j'\geq j-3}\|\Delta_{j'}\tilde u\|_{L^{p_1}}
 \|S_{j'+2}\dj v\|_{L^\infty}.
 $$
 Hence this term may be bounded as desired (just follow the previous computations). 
 \smallbreak
 In order to handle the second term of \eqref{est:com5}, we  write that, 
 according to Bersntein's inequality, 
 $$ \|\d_k\Delta_{j'}\tilde u\|_{L^{p_1}}
 \|S_{j'+2}\dj v\|_{L^\infty}\leq C2^{j'} \|\Delta_{j'}\tilde u\|_{L^{p_1}}
 \|S_{j'+2}\dj v\|_{L^\infty}.
 $$
 Hence 
 $$
 2^{j(s-1)} \sum_{j'\geq j-3}\|\d_k\Delta_{j'}\tilde u\|_{L^{p_1}}
 \|S_{j'+2}\dj v\|_{L^\infty}\leq C
  \sum_{j'\geq j-3}2^{(j-j')(\nu+\frac{d}{p_1}-1)}c_{j'}
\|\tilde  u\|_{B^{\frac{d}{p_2}+\nu}_{p_2,r_2}}d_j\|v\|_{B^{s-\nu-\frac{d}{p_1}}_{\infty,r_1}},
$$
 which leads to the desired inequality provided that $\nu+\frac{d}{p_1}-1>0.$
  \qed

\bibliographystyle{plain}
\bibliography{LowMach}

\end{document}